\documentclass[a4paper,twoside,11pt]{article}
\usepackage{wiaspreprint}
\usepackage{amsmath,amssymb,amsthm,mathptmx,color}
\numberwithin{equation}{section}
\newtheorem{theorem}{Theorem}
\newtheorem{remark}[theorem]{Remark}
\newtheorem{definition}[theorem]{Definition}
\newtheorem{proposition}[theorem]{Proposition}
\newtheorem{lemma}[theorem]{Lemma}

\newtheorem{assumption}[theorem]{Assumption}

\newcommand{\f}[1]{\pmb{#1}}

\DeclareMathOperator{\R}{\mathbb{R}}

\DeclareMathOperator{\C}{\mathcal{C}}

\DeclareMathOperator{\He}{\f{H}^1}

\DeclareMathOperator{\M}{\mathcal{M}}

\DeclareMathOperator{\ra}{\rightarrow}
\DeclareMathOperator{\de}{\text{d}}

\DeclareMathOperator{\cur}{curl}

\DeclareMathOperator{\out}{out}

\DeclareMathOperator{\supp}{supp}

\DeclareMathOperator{\sogn}{sign}
\newcommand{\sign}{\sogn\hspace*{-0.3em}^+\hspace*{-0.2em}}
\DeclareMathOperator{\esssup}{ess\,sup}
\DeclareMathOperator{\essinf}{ess\,inf}

\DeclareMathOperator{\di}{\nabla \cdot}

\newcommand{\ov}[1]{\overline{#1}}

\DeclareMathOperator{\curl}{\nabla \times}

\renewcommand{\t}{\partial_t  }

\newcommand{\un}[1]{\underline{{#1}}}

\newcommand{\n}{\nabla}

\definecolor{lightgray}{RGB}{174,205,238}

\newcommand{\tet}{\tilde{\theta}}
\newcommand{\tz}{\tilde{z}}
\newcommand{\tA}{\tilde{\f A}}
\newcommand{\tu}{\tilde{\f u}}

\begin{document}
\author[D. H\"omberg, R. Lasarzik]{Dietmar H\"omberg\footnote{Technische
  Universit\"at Berlin\\ Institut
  f\"ur Mathematik \\ Str. des 17. Juni 136 \\
  10623 Berlin\\ Germany}\addmark{~\,,}{}\footnote{Department of \\
  Mathematical Sciences\\
  NTNU \\ Alfred Getz vei 1\\ 7491
Trondheim\\ Norway}\addmark{~\,,}{}\footnote{Weierstrass Institute \\
Mohrenstr. 39 \\ 10117 Berlin \\ Germany \\
E-Mail: dietmar.hoemberg@wias-berlin.de\\
\hphantom{E-Mail:} robert.lasarzik@wias-berlin.de},
Robert Lasarzik\addmark{~}{3}}

\title[Weak entropy solutions, existence and weak-strong uniqueness]{Weak entropy solutions to a model in induction hardening, existence and weak-strong uniqueness
}	
\nopreprint{2671}	
\selectlanguage{english}		
\thanks{Partially supported by the European Union's Horizon 2020 research and innovation programme under
the Marie Sk\l{}odowska-Curie grant agreement No. 675715 (MIMESIS)}				
\maketitle
\begin{abstract}
In this paper, we investigate a 
model describing induction hardening of steel. The related system consists of an energy balance, an ODE for the different phases of steel, and Maxwell's equations in a potential formulation.
The existence of weak entropy solutions is shown  by a suitable regularization
and discretization technique. Moreover, we prove the weak-strong uniqueness of
these solutions, i.e., that a weak entropy solutions coincides with a classical solution emanating form the same initial data as long as the classical one exists. 
The weak entropy solution concept has advantages in comparison to the previously introduced weak solutions, e.g., it allows to include free energy functions with low regularity properties corresponding to phase transitions.
%
%
\end{abstract}

\setcounter{tocdepth}{2}
\tableofcontents
\section{Introduction and main results}\label{sec:intro}
\textit{Induction hardening} is an energy efficient method for the heat treatment of\textit{ steel parts}. 
The general goal of these surface heat treatments is to adapt the boundary layer of a work piece to its special areas of application in terms of hardness, wear resistance, toughness, and stability. 
To achieve this goal, the boundary region of the steel component  is first
heated and afterwards cooled to influence the micro-structure in such a way that
the boundary layer emerges in the so-called \textit{martensitic} phase. The
resulting phase mixture in the component then exhibits the most desirable
structural properties combining a hard boundary layer to reduce
abrasion with a softer interior to reduce fatigue effects. 

In induction hardening, the heating of the work piece is caused by a periodically varying electromagnetic field. The varying magnetic flux in the inductor coil in turn induces a current in the work piece, which is (due to resistance) partly transformed into eddy current losses resulting in Joule heating.
Heated up to a certain temperature regime, the high-temperature solid steel phase, called \textit{austenite}, emerges. Afterwards, the material is cooled rapidly to transform the austenite into the desired martensite phase.

The \textit{skin effect}  describes the tendency of an alternating electric
current to induce a current density that is largest near the outer surface of
the work piece. In a cylinder, the heating occurs mainly between the surface and
the so-called skin depth, which is mainly determined by the material and the frequencies of the current source. But for a work piece with a more complex geometry, this relation becomes more involved. 

For example, gears are treated with multiple frequencies to achieve a hardening of the tip as well as the root of the gear (see~\cite{elisabetta} or~\cite{schwenk}).  To predict the outcome of the heat treatment, numerical simulations proved as an important tool~\cite{simulation}. 
Therefore and also for controlling the process optimally, a mathematical understanding of the procedure is essential.

To model this interplay of electric fields with high frequencies,  eddy currents, steel phases, and heat conduction, a mathematical model is given relying on an energy balance with Joule heating, Maxwell's equation in a potential formulation, and an ODE for the different  phases evolving in the steel. 
A special difficulty arises since the terms of highest differential order in time and space are nonlinearly coupled in all equations, which poses severe difficulties for the analysis. 
In~\cite{hoemberg},~\cite{elisabetta}, and~\cite{surfacehardening}, this model is investigated in the context of weak solutions (see~\cite{hoemberg}) and stability (see~\cite{elisabetta}). 
 The novelty of the work at hand is the introduction of the concept of weak-entropy solutions into the context of induction hardening. 
In contrast to weak solutions, this concept allows for  physical  more realistic
assumptions (compare the temperature dependence of the heat conductivity and the
electric conductivity to the model considered in~\cite{elisabetta}).  In
comparison to the weak-solution concept, the weak formulation of the energy
balance is replaced by the  energy inequality and the entropy
production rate. This is similar to the solution concept for the Navier--Stokes--Fourier system~\cite{singular} or a system with phase transitions, e.g., in~\cite{Giulio}. In comparison to the weak concept in~\cite{hoemberg}, the weak entropy solution concept allows to show the weak-strong uniqueness of solutions, i.e., that they coincide with a local strong solution emanating from the same initial data as long as the latter exists. In recent publications, it was proved that weak solutions may give rise to nonphysical non-uniqueness (see~\cite{Isett} and~\cite{buckmaster}) such that weak solutions may not be the most favourable solution concept. 

Another advantage of the introduced solution concept is  that the function modeling the latent heating in the energy balance can be chosen less regular compared to the previous definitions of weak solutions. 
The latent heat can be modelled based on the Heaviside function corresponding to the sharp phase transition behaviour one would expect in the evolution of the steel phases.
Additionally, we want to mention that the solutions concept is in line with the physics of the systems. Such a thermodynamical system is derived from the thermodynamical principles and these are the heart of the solution concept. 
Instead of formulating the energy balance in a weak sense, we rather formulate the Clausius-Duhem inequality and the energy conservation in a weak sense, which are the underlying thermodynamical principles. 

As already mentioned, the induction hardening of components with a complex geometry depends on induced alternating currents with different frequencies. These frequencies can become rather high, posing severe problems for the numerical solution, if one tries to solve the system in the time domain, which requires a very fine discretization. Hence, one would be interested in a singular limit  for this system (see~\cite{homo}). 
The relative energy inequality is a suitable tool to proof rigorously~\cite{singular} the convergence  of solutions of the systems to solutions of the singular limit system. We will explore this for the introduced solution concept in a future publication.


The paper is organized as follows,
in the next section (Sec.~\ref{sec:model}) the model is introduced with a special focus on its thermodynamic consistency and the main results of the paper are collected. In Sec.~\ref{sec:ex}, the existence proof is executed. 
We start by stating the regularized discretized system, and show its solvability (see Sec.~\ref{sec:ex}). 
Deriving \textit{a priori} bounds (see Sec.~\ref{sec:firstapri} and Sec.~\ref{sec:apri2}), we are able to go to the limit with the regularization and discretization in a suitable manner (see Sec.~\ref{sec:convfirst} and Sec.~\ref{sec:conv2})  to prove the existence of weak entropy solutions in the end. Afterwards, we comment on the adaptations of the proof under stronger assumptions~\ref{sec:weakss}.
The weak-strong uniqueness is proven in Sec.~\ref{sec:weakstrong}, defining the
relative energy and the relative dissipation for the system, we give some
preliminary results (see Sec.~\ref{sec:relen}), which helps us to prove the
relative energy inequality (see Sec.~\ref{sec:relin}). The dissipative terms are
handled (see Sec.~\ref{sec:diss}) and the remaining terms are estimated
appropriately (see Sec.~\ref{sec:est}) to prove the weak-strong uniqueness in the end.

\subsection{Model\label{sec:model}}

The model for the induction hardening process is based on Maxwell's equations
\begin{align}
\curl \f H = \f J+ \t \f D  \,, \qquad \curl \f E = - \t  \f B \,, \qquad \di \f B = 0\,,\qquad \di \f D = \rho  \,, \label{eq:max}
\end{align}
where $\f E$ is the electric field, $\f B$ the magnetic induction, $\f H$ the magnetic field,  $\f J$ the spacial current density, and $\f D$ the electric displacement field.  
The free  electric charge potential is denoted by $\rho$. 
Then, we assume Ohm's law and a linear relation between the magnetic induction and the magnetic field as well as between the electric and the electric displacement field:
\begin{align}
\f J= \sigma 
 \f E\,, \qquad \f B = \mu_0 
 \f H\,,\qquad \f D = \varepsilon \f E \,, \label{linrel}
\end{align}
where $\sigma$ denotes the electrical conductivity, $\mu $ the magnetic permeability, and $ \varepsilon$ the electric  permittivity. 
As usual for induction phenomena, the magneto quasistatic approximation of Maxwell's equations is considered, i.e., $ \t \f D \equiv0$. 
In view of~\eqref{eq:max}$_3$, a magnetic vector potential $\f A$ can be introduced via 
\begin{align}
\f B = \curl \f A\, , \qquad \di \f A = 0  \quad \text{in }  D\,.\label{introA}
\end{align}
From~\eqref{eq:max}$_2$ and~\eqref{introA}, we observe the existence of a scalar potential~$\phi$ such that 
\begin{align}
\f E+ \t \f A =  \nabla \phi    \,,   \quad\text{and}\quad \f J =-  \sigma
(\t \f A - \nabla \phi)  \quad    \text{in }D  \times (0,T) \,,\label{jouleheat}
\end{align} 
where we used~\eqref{linrel}$_1$.  
Thus, from the relations~\eqref{eq:max}$_1$,~\eqref{linrel}$_2$,~\eqref{introA}$_1$, and defining the source current density $\f J_s$ by $\f J_s:= \sigma (\theta) \nabla \phi$, we obtain the following formulation of Maxwell's equations suitable for our purposes 
\begin{align}
\sigma(\theta ) \t \f A + \curl \left ( \frac{1}{\mu(z) } \curl\f A\right ) = \f J_s \quad\text{in }D \,, \qquad \f n \times \f A = 0 \quad \text{on }\partial D  \,.\label{maxwellerste}
\end{align}

\begin{figure}[h] 
   \centering
   \includegraphics[width=2in]{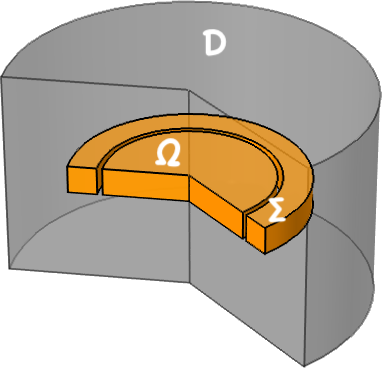} 
   \caption{The setting: hold-all domain $D$, workpiece $\Omega$ and inductor $\Sigma$.}
   \label{fig:setting}
\end{figure}

We use the following setting: we consider the domain $D$, consisting of  the work piece, which is assumed to fill the volume $\Omega$, the inductor called $\Sigma$, and air, which is assumed to fill $D/(\Omega \cap \Sigma )$. 
Maxwell's equation is assumed to be fulfilled in the whole domain $D$ and $\f J_s$ is supported in $\Sigma$, i.e., $\supp(\f J_s) = \Sigma$ (see Figure \ref{fig:setting}).
The properties of the material constants $\sigma$ and $\mu$ are allowed to depend on the respective domain, i.e., 
\begin{align*}
    \sigma(\f x ,\theta ) 
:=     \begin{cases}
\sigma( \theta )  & \text{in }\Omega \,, \\
\sigma_{\text{cond}} & \text{in }\Sigma\,,\\
0 & \text{in } D /(\Omega \cap \Sigma ) \,.
\end{cases}
\qquad \mu ( \f x ,z ) = \begin{cases}
\mu( z) & \text{in }\Omega \,, \\
\mu_{\text{cond}} & \text{in }\Sigma\,,\\
\mu_{\text{a}}  & \text{in } D /(\Omega \cap \Sigma ) \,.
\end{cases}
\end{align*}
Here, we assume that there exists a $\un\sigma>0$ and a $\un\mu>0$ such that $\un\sigma \leq \sigma(\theta) $ for all $\theta\in[0,\infty)$ and $\un \sigma \leq \sigma_{\text{cond}}$ as well as $\un\mu\leq \mu(\f x , z) $ for all $ \f x \in D$ and $z \in \R$. For convenience, we define $\sigma_{\out} = \sigma_{\text{cond}} $ and $\mu_{\out}= \mu_{\text{cond}} $  in $\Omega$ as well as $\sigma_{\out} = 0 $ and $\mu_{\out} = \mu_{a}$ in $D/(\Omega\cap \Sigma)$.
In the sequel we will drop the $x$ - dependency of $\mu$ and $\sigma$ to simplify the exposition.

The energy balance is only considered in $\Omega$. 
%
%
%
%
The evolution of the internal energy is determined by the heat flux and the Joule heating (see~\cite[Sec.~3.8, Eq.~3.8.22]{maugin}):
\begin{align}
\t e +\di \f q = \f J \cdot \f E   \,,\label{energybalance}
\end{align}
where 
$e$ denotes the internal energy of the system and $\f q$ the heat flux, which is  given by  a generalization of  Fourier's law 
\begin{align*}
\f q = - \kappa(\theta , z ) \nabla \theta  \,, \qquad \text{with } \kappa (\theta ,z) \in [\underline{\kappa}, \overline{\kappa}]  \quad \forall (\theta ,z) \in (0,\infty)\times \R\,.
\end{align*}
The Helmholtz free energy $\psi$, with $\psi = \psi( \theta,z)$,  is defined via $\psi = e - \theta s $, where $s$ denotes the entropy . The Clausius--Duhem inequality $ \t s + \di ( \f q / \theta) \geq 0$ can be transformed 
to
\begin{align*}
\theta \left (\t s  + \di \left (\frac{\f q}{\theta } \right )\right ) = {}&\t e + \di \f q - \partial_t \psi - \t \theta s  - \frac{\f q  \cdot \nabla \theta }{\theta }
\\
={}& - ( \psi_\theta + s) \t \theta + \f J \cdot \f E  - \psi_z \t z + \frac{\kappa(\theta ,z) | \nabla \theta |^2 }{\theta }\\
 ={}& - ( \psi_\theta + s) \t \theta  + \sigma(\theta ) | \t \f A |^2  - \psi_z \t z + \frac{\kappa(\theta ,z) | \nabla \theta |^2 }{\theta }\\
\geq{}& 0 \,.
\end{align*}
It is satisfied for all choices of variables, if the standard relation $\psi_\theta + s = 0$ and the phase field equation $\tau(\theta) \t z + \psi_z ( \theta ,z) =0$ are satisfied (see~\cite{elisabetta}). 
By the definition of the Helmholtz free energy and using the phase equation, we find
\begin{align*}
\t e  =  \partial_t \psi - \t \theta  \psi_\theta - \theta  \partial_t \psi_\theta = \psi _ \theta \t \theta  + \psi_ z \t z   - \t \theta  \psi_\theta - \theta  \partial_t  \psi_\theta 
= - \theta  \partial_t \psi_\theta  - \tau(\theta ) (\t z)^2 \,.
\end{align*}
Inserting this back into~\eqref{energybalance} and using similarly the expressions~\eqref{jouleheat} and Fourier's law, we end up with equation~\eqref{temp}.
All in all, we observe that the equations of motion are given by
\begin{subequations}\label{eq:strong}
\begin{align}
- \theta \t \psi_\theta  - \di ( \kappa ( \theta , z ) \nabla \theta ) ={}& \sigma ( \theta ) | \t \f A   |^2 + \tau (\theta ) | \t z|^2\quad\text{in }\Omega \,,\label{temp}\\
\sigma (\theta )  \t \f A  + \curl \frac{1}{\mu(z) } \curl \f A  ={}& \f J_s \quad \text{in }D\,,\label{magnetic}\\
\tau(\theta) \t z + \psi_z(\theta , z )  = {}& 0 \quad \text{in }\Omega \,\label{inner}
\end{align}
equipped with the boundary conditions 
\begin{align}
\f n \cdot \kappa(\theta ) \nabla \theta = 0 \quad \text{on }\partial\Omega \,, \qquad \f n \times \f A = 0 \quad \text{on } \partial D \label{boundarycond}
\end{align}
and initial conditions
\begin{align}
\theta ( 0) = \theta_0\quad \text{in }\Omega \,,\qquad \f A (0) = \f A_0 \quad \text{in }D  \,,\qquad z (0) = z_0 \quad \text{in }\Omega \,.\label{initialcond}
\end{align}
\end{subequations}
\begin{remark}
An example for a typical free energy $ \psi$ is given by
\begin{align}
\psi(\theta ,z) =- \theta (\log \theta -1) + ( z_{\text{eq}}(\theta ) - z)^2_+\label{example}
\,,
\end{align}
where $ (x)_+$ is defined via $ ( x) _+ : = \max\{ 0 , x\} $ and $ z_{\text{eq}}\in \C^{1,1}\cong W^{2,\infty}$ with some appropriate growth conditions as $\theta\ra \infty$.    Note that $\psi $ is not twice continuously differentiable, but only twice weakly differentiable. 
The choice of such an irregular function corresponds to the fact that the metal is subjected to a kind of Hysteresis effect when it changes the phase, it should not immediately change back even, if the temperature changes back again. This is modeled by the nonsmooth function $ (\cdot )_+$. 
\end{remark}

\subsection{Hypothesis and preliminaries}
We assume the domain $\Omega $ to be of class $\C^2$ and simply connected. This can be  generalized in the limit, we exclusively use the regularity in the approximate setting. 
We use standard definitions of Lebesgue and Sobolev spaces. Additionally, we define the space 
\begin{align*}
\f H({\cur},D) := \{ \f u \in L^2(D )^3 : \curl \f u \in L^2 (D  )^3 \} \,,
\end{align*}
which is equipped with the graph-norm. Additionally, we define
 \begin{align*}
 \f H_{\cur} := \{ \f u \in L^2(D )^3 : \curl \f u \in L^2 (D  )^3 \,, \quad \di \f u =0 \text { in }D   \,, \quad \f n \times \f  u  = 0 \text{ on }\partial D    \}\,,
 \end{align*}
equipped with the norm $ \| \curl \cdot \|_{ L^2}$, where the operators and the traces have to be interpreted in a weak sense (see~\cite{trace} for the definition of the suitable trace operators on Lipschitz domains).  
On $\f H_{\cur}$, this norm is equivalent to the usual $\f H^1(\Omega)^3$-norm (see for instance~\cite{amrouche} for the associated estimates or~\cite{wahl} for the connection of the associated estimate to the topological properties (Betti number) of the domain).

We collect the different assumptions for the different results of the paper:
\begin{assumption}\label{Asume:1}
The free energy $\psi $ is a function given by 
$\psi : [0,\infty) \times \R \ra \R$. 
It should fulfill, $ \psi \in \C([0,\infty)\times \R ; \R) $,   $ \psi \in \C^{1,1}((0,\infty)\times \R;\R) $, and $ \psi (0,z) \geq 0$ for all $z\in\R$. 
Additionally, we assume physical convenient conditions on $\psi$, i.e.,
\begin{align}
 | \psi _{z}|\leq C \, , \quad  | \psi _{zz} |\leq C \, , \quad   
  c \leq - \theta \psi_{\theta\theta } \leq C \, , \quad  |(1+\theta) \psi_{z\theta}| \leq C \quad \text{for all }\theta,\,z \in [0,\infty)\times \R \,.\label{heatcapacity}
\end{align}
For the head conduction, we assume that $\kappa$ depends on both variables multiplicative, meaning that there exist two functions $ \tilde{\kappa}$, $\bar{\kappa}: \R \ra [0,\infty) $ such that
$\kappa(\theta , z) = \tilde{ \kappa}(\theta) \bar{\kappa}(z)$.

The functions $\tilde{\kappa} $, $\bar{\kappa}$, $\tau$,  $\mu$, $ \tau'$,and $\partial_{z}\mu$ are bounded from below and above and continuous in $\theta $ and $z$, respectively, i.e., 
\begin{align*}
\tilde{\kappa}\,, \bar{\kappa }\in \C(\R;[0,\infty) \text { and } \mu,\, \tau  \in\C^1(\R;\R)  \text{ with } 0 < c \leq \tilde\kappa\, \bar{\kappa}\,, 
\tau\, , \mu \leq C \text { and }|\partial_z \mu(z)|+| \theta  \tau'(\theta ) |
 \leq c    \,
\end{align*}
for all $(\theta ,z) \in [0,\infty)\times \R$.  For $\sigma$, we assume that $\sigma(\f x , \theta) \geq c $ for all $\f x \in \Omega \cap \Sigma$ and $\theta \in [0,\infty)$. 


\end{assumption}
\begin{remark}
The term $ - \theta \psi_{\theta \theta}$ is often referred to as heat-capacity and the assumptions on its boundedness from below and above are rather standard (compare with~\cite{elisabetta} or~\cite{tomassetti}). 
Note that since $\psi \in \C^{1,1}((0,\infty )\times \R; \R )$ it is twice weakly differentiable with essentially bounded second derivative, i.e., $ \psi \in W^{2,\infty}((0,\infty)\times \R; \R)$. 
\end{remark}
The free energy $\psi $ is a function given by 
$\psi : (0,\infty) \times \R \ra \R$. 
Under the next set of assumptions, the weak-strong uniqueness of solutions is proved. 
\begin{assumption}\label{Assume:2}
Let Assumptions~\ref{Asume:1} be fulfilled. Additionally, the free energy function should fulfill several standard assumptions: 
it should be 
concave in $\theta$ (see~\eqref{heatcapacity}) and 
convex in $z$ such that there exists a $c>0$ with
\begin{align}
  \psi _{zz}\geq c >  0  \label{phi:concave}
\end{align}
for all $ \theta , z \in (0,\infty) \times \R$. 
Furthermore, 
$ \psi \in \C^3((0,\infty)\times \R; \R)$ with bounded third derivatives. 
For the function~$\mu$, denoting the magnetic permeability, we assume 
 \begin{align}
 \mu \in \C^{2}(\R) \,, \quad  0 \leq-( \mu^2(y))'' \leq  C_{\mu}
 \text{ for all }y\in\R .\label{condmu}
 \end{align}
 The functions $\sigma$, $\tau$ and $\kappa$ are additionally assumed to be continuously differentiable, i.e., $ \kappa \in \C^1((0,\infty)\times \R;\R)$ and $\sigma$, $ \tau\in \C^1((0,\infty); \R)$ and this derivative should vanish sufficiently fast at infinity, i.e.,
 \begin{align}
\left (  | \nabla_{(\theta ,z)} \kappa(\theta ,z) | + | \partial_\theta \sigma(\theta ) | + | \tau'(\theta)| \right ) (\theta^3+  \theta +1) \leq c \, \label{boundderivative}
 \end{align}
for all $ \theta , z \in (0,\infty) \times \R$.
\end{assumption}
 The assumption of concavity of $\mu^2$ is stronger than the concavity of $\mu$ itself. Indeed for $\mu>0$, we observe that  in the equality $(1/2)( \mu^2)'' =\mu''\mu + ( \mu')^2 $ the second term on the right-hand side is positive.  
 This concavity assumption is not optimal, but some assumption stronger than concavity of $\mu$ is of need for our technique. We settled for this special assumption due to its readability.

%

\subsection{Weak entropy solutions and main results}
\begin{definition}\label{def:weak}
\begin{subequations}
The functions $( \theta , z,\f A)    $ are called weak entropy solution  to~\eqref{eq:strong}, if they fulfill
\begin{align}
\begin{split}
\theta \in{} &  L^\infty(0,T; L^1(\Omega)) \text{ with } \theta \geq 0 \text{ a.e. in }\Omega \times (0,T)\text{ and } \\\log \theta \in{}& L^2(0,T; W^{1,2}(\Omega)) \cap L^\infty(0,T; L^1(\Omega))\,,\\
z \in {}& W^{1,\infty}(0,T; L^\infty(\Omega)) \cap L^2(0,T;W^{1,2}(\Omega))\,, \\
\f A \in {}& L^\infty(0,T; \f H_{\cur}) \cap W^{1,2} ( 0,T; L^2(\Omega\cup \Sigma )) 
\end{split}\label{regweak}
\end{align}
and the  entropy inequality (or entropy production rate)
\begin{multline}
\int_\Omega  \psi_\theta( \theta(t) , z(t) )   \vartheta  (t) \de x + \int_0^t \int_\Omega \vartheta \left ( \kappa (\theta ,z) {|\nabla \log \theta |^2} + \sigma (\theta ) \frac{|\t \f A |^2}{\theta }+ \tau(\theta ) \frac{|\t z|^2}{\theta } \right ) \de x \de s  \\
- {} \int_0^t \int_\Omega \kappa(\theta ,z) {\nabla \log \theta } \cdot \nabla \vartheta \de x \de s \leq \int_\Omega \psi_\theta ( \theta_0 ,z_0 ) \vartheta
(0)  \de \f x+  \int_0^t \int_\Omega \psi_\theta (\theta ,z) \t \vartheta \de x \de s \,
\label{entropy}
\end{multline}
for all $\vartheta \in \C([0,T]; L^\infty(\Omega)) \cap L^2 (0,T ; \He(\Omega)) \cap W^{1,1}(0,T ; L^\infty(\Omega))$ with $\vartheta \geq 0$ a.e. in $\Omega \times(0,T)$ and for almost every $t\in(0,T)$,
the weak formulation for the inner variables
\begin{align}
\int_0^T \int_D ( \sigma( \theta ) \t \f A  , \varphi ) + \left ( \frac{1}{\mu(z)} \curl \f A , \curl \varphi\right ) \de x \de t = \int_0^T \int_{\Sigma}  \f J_s \cdot \varphi \de \f x \de t   \,,\label{weakA}
\end{align}
\begin{align}
 \tau(\theta ) \t z + \psi _z ( \theta , z  )= 0 \quad
  \text{ a.e. in } \Omega \times (0,T)   \label{weakz}
\end{align}
for all $ \varphi \in L^2( 0,T; \f H _{\cur})$,
and the energy inequality 
\begin{multline}
 \left .\int_\Omega e(\theta(s) ,z(s))   \de \f x \right |_0^t+ \left .\int_D \frac{1}{2\mu(z) } | \curl \f A |^2   \de \f x\right |_0^t + \int_0^t \int_\Omega  \frac{\mu'(z) \t z }{2\mu^2(z)} | \curl \f A |^2 \de \f x \de s \\+ \int_0^t \int_{D/\Omega} \sigma _ {\out} | \t \f A |^2 \de \f x \de s\leq {}\int_0^t \int_\Sigma \f J_s \cdot \t \f A \de \f x \de s \,
\label{energyequality}
\end{multline}
for almost every $t\in(0,T)$, where $ e $ is given by $ e = \psi - \theta \psi_\theta$.
The initial values in the equations~\eqref{weakA} and~\eqref{weakz} are attained $ \f A(0) = \f A_0 $ in $ \f H_{\cur}$ and $ z(0)= z_0$ in $ L^\infty$.  
\end{subequations}
\end{definition}
\begin{remark}
The initial values are attained in a weak sense, i.e., from the regularity assumption we infer $ z \in \C([0,T]; L^\infty) $ and $\f A   \in \C_w ([0,T]; \f H_{\cur} )$.
\end{remark}
\begin{theorem}\label{thm:weakin}
Let $\Omega $ be of class $\C^2$ and let Assumption~\ref{Asume:1} be fulfilled.  To every $ \theta _0 \in L^1( \Omega) $ with $ \essinf_{\f x \in \Omega} \theta_0   > 0$ and $ \log \theta _0\in L^1(\Omega)$, $ \f A_0 \in \f H_{\cur} $, $z_0\in L^\infty (\Omega)$, and $ \f J_s \in L^2(0,T;L^2(\Sigma )) $,
 there exists a global weak solution to~\eqref{eq:strong} in the sense of Definition~\ref{def:weak}. 
 
\end{theorem}
\begin{remark}\label{rem:add}
If we assume in addition that $ \psi _{\theta \theta z} $ exists and is bounded by
\begin{align}
 |\theta ^s \psi_{\theta,\theta , z}| \leq c  \text { for } s > 1/3 \,, \label{psithetas}
\end{align}
then we infer in addition that  
\begin{align}
\theta \in{} L^q (0,T; W^{1,q}( \Omega))\cap L^p(0,T; L^p(\Omega)) \label{weakeqreq}
\end{align}
for $ q \in [1,4/3)$ and $p\in [1,5/3)$.

\end{remark}

%
%
%
%

\begin{remark}
In the case, that Assumption~\ref{Asume:1} as well as~\eqref{psithetas} holds and additionally $\mu $ is constant, i.e., $ \partial _z \mu \equiv 0$, we could even show that the weak formulation of the energy balance, 
\begin{align}
-\int_0^T \int_\Omega e( \theta ,z)    \t \zeta \de \f x \de t  + \int_0^T \int_\Omega \kappa (\theta ,z) \nabla \theta \cdot \nabla \zeta \de \f x \de t = \int_0^T \int_\Omega \sigma (\theta) | \t \f A |^2 \zeta \de \f x \de t \,,\label{weakenergybalance}
\end{align}
holds for all $ \zeta \in W^{1,5/2}( 0,T; L^{5/2}(\Omega)) \cap L^4(0,T;W^{1,4}(\Omega))\cap L^\infty (0,T; L^\infty(\Omega))$ (compare to~\cite{hoemberg}).
\end{remark}
%

\begin{definition}\label{def:regular}
The functions $( \theta , z,\f A)    $ are called regular solution to~\eqref{eq:strong}, if they are a weak entropy solution with respect to Definition~\ref{def:weak} and additionally fulfill the entropy inequality~\eqref{entropy} as an equality and admits the regularity
\begin{align}
\begin{split}
\log \theta \in{}& W^{1,1}(0,T; \C(\ov \Omega)) \cap L^2(0,T; \C^{1}(\ov\Omega)) \text{ with } \di ( \kappa ( \theta , z) \nabla \log \theta ) \in L^1(0,T; \C(\ov\Omega))\text{ and }\\ \theta(\f x ,t) \geq{}& 0 \text{ a.e.~in } \Omega \times (0,T)\\
\f A \in{}& \C  ([0,T];C^{1}(\ov \Omega)) \cap W^{1,1}(0,T;\C^1 (\ov\Omega))\,,.
\end{split}\label{regregular}
\end{align}
\end{definition}
\begin{theorem}\label{thm:weakstrong}
Let  
Assumption~\ref{Assume:2} be fulfilled. If there exists a regular solution $(\tet,\tz,\tA)$ according to Definition~\ref{def:regular} and a weak entropy solution $(\theta,z,
\f A)$ according to Definition~\ref{def:weak} emanating form the same initial data, then it holds
\begin{align*}
\theta =\tet\,, \quad 
z =\tz\,, \quad \f A = \tA\,.
\end{align*}
\end{theorem}
\begin{remark}
The standard example~\eqref{example} fulfills the Assumption~\ref{Asume:1} such that the solution concept due to Definition~\ref{def:weak} seems to be appropriate for this relevant functional. 
Additionally, we observe that the weak-strong uniqueness result does not hold for the example~\eqref{example}. It only holds, if the material undergoes no phase transition. This seems to be reasonable since phase transitions modeled by Hysteresis are known to induce bifurcations (see for example~\cite{bifur}). 
\end{remark}

\subsection{Preliminaries}
In this section, we present  the energy equality and bounds of the inner energy and the entropy of the system. 
\subsubsection{Energy equality}
The energy inequality corresponds to the physical principle of energy conservation and is an important feature of the system providing \textit{a priori} estimates.
In the following, we proof it formally.

\begin{lemma}
Let $(\theta , z,\f A)$ a sufficiently smooth solution to~\eqref{eq:strong}. Then the energy inequality holds
\begin{multline*}
 \left .\int_\Omega e(\theta(s) ,z(s))   \de \f x \right |_0^t+ \left .\int_D \frac{1}{2\mu(z) } | \curl \f A |^2   \de \f x\right |_0^t + \int_0^t \int_\Omega  \frac{\mu'(z) \t z }{2\mu^2(z)} | \curl \f A |^2 \de \f x \de t \\+ \int_0^t \int_{D/\Omega} \sigma _ {\out} | \t \f A |^2 \de \f x \de s= {}\int_0^t \int_\Sigma \f J_s \cdot \t \f A \de \f x \de s \,
\end{multline*}
for all $t\in [0,T]$, where $ e $ is given by $ e = \psi - \theta \psi_\theta$. 
\end{lemma}

Testing equation~\eqref{temp} formally with $1$, equation~\eqref{magnetic} with $\t \f A $, equation~\eqref{inner} with $\t z$, and adding them up leads to
\begin{align}
\begin{split}
- \int_\Omega \theta  \partial _t \psi_\theta ( \theta ,z ) + \di (\kappa ( \theta , z) \nabla \theta ) -\left  ( \frac{1}{\mu(z)} \curl \f A \cdot \curl \t \f A \right ) - \psi_z ( \theta , z  ) \t z \de \f x &
\\
\frac{\de }{\de t}\int_{D/\Omega} \frac{1}{2\mu_{\out} }| \curl \f A |^2 \de \f x + \int_{D/\Omega} \sigma_{\out} | \t \f A |^2 \de \f x 
&= \int_\Sigma \f J_s \cdot \t \f A \de \f x  \,.
\end{split}
\label{energycalc}
\end{align}
For the heat-flux, we observe due to the Neumann boundary conditions that
\begin{align*}
\int_\Omega \di ( \kappa(\theta , z) \nabla \theta ) \de \f x = \int_{\partial \Omega} \f n \cdot( \kappa (\theta , z)  \nabla \theta ) \de S = 0\,.
\end{align*}
Additionally,  we observe that 
\begin{align*}
\t e ( \theta , z) = \t ( \psi(\theta ,z) - \theta \psi_\theta ( \theta ,z) )  = \psi_z( \theta,z) \t z - \theta \t \psi _{ \theta} \,
\end{align*}
and 
\begin{align*}
\frac{1}{\mu(z)} \curl \f A \cdot \curl \t \f A  = \frac{1}{2\mu(z)} \t |  \curl \f A|^2  = \t \left ( \frac{1}{2\mu(z)} | \curl \f A |^2 \right ) + \frac{\mu'(z) \t z }{2 \mu^2(z)}| \curl \f A |^2  \,,
\end{align*}
which implies the assertion.
\begin{lemma}\label{lem:thetae}
Let the Assumption~\ref{Asume:1} be fulfilled. Then there exist two constants $0<c \leq C <\infty
$ such that 
\begin{align*}
c (\theta-1) \leq e( \theta ,z ) = \psi(\theta , z) - \theta \psi_{\theta}(\theta ,z) \leq C ( \theta+1) \quad \text{and} \quad c ( \log \theta-1) \leq - \psi_\theta (\theta ,z) \leq C( \log \theta+1) 
\end{align*}
for all $\theta , \,z \in (0,\infty) \times \R$. 
Additionally, the two inverse functions $ \hat{e}: (0,\infty) \times \R \ra (0,\infty) $ and  $ \hat s: \R \times \R \ra \R $ defined via $ \hat e ( e (\theta ,z) , z) = \theta $ and $\hat{s}(\psi_\theta (\theta ,z) ,z) = - \log \theta $ are well defined and differentiable. 
Additionally, it holds $ e( \theta , z) \geq 0$ for all $ (\theta, z)  \in [0,\infty) \times \R$.
\end{lemma}
\begin{proof}
The derivative of $e$ with respect to $\theta$ is given by $ - \theta \psi_{\theta \theta } $ and controlled from below and above due to Assumption~\eqref{heatcapacity}. 
Writing
\begin{align*}
e(\theta ,z) = e(0,z) + \int_0^\theta \partial _\theta e(r,z) \de r = e(0,z) + \int_0^\theta \left ( - r \psi_{\theta \theta} (r ,z )\right ) \de r \geq  e ( 0 ,z) + c   \theta\quad ( \leq e ( 0 ,z) +C   \theta) \,.
\end{align*}
The good control on the derivative of $\theta \mapsto e( \theta ,z)$ allows to deduce the existence and regularity of an inverse function with the inverse function theorem. 

Now we observe that $ \psi_\theta ( \theta , z) = \psi _\theta (1/ \exp({-\log \theta }) , z)$ such that for $ \tilde{\psi }_\theta (\zeta , z)  := \psi_\theta (\exp({-\zeta }), z)$ it holds 
\begin{align*}
\partial _\zeta \tilde{ \psi}_\theta (\zeta ,z) =- \exp({-\zeta}) \psi_{\theta \theta }  ( \exp({-\zeta}) , z)  \,.
\end{align*} 
The same arguments as for the function $e$ show that
\begin{multline*}
\psi_\theta (\theta , z) =  \psi_{\theta } ( 1,z) + \int_0^{-\log \theta } \left ( - \exp({-\zeta}) \psi_{\theta \theta } ( \exp({-\zeta }) , z)\right )  \de \zeta  \geq  \psi _{\theta} (1,z) - c \log \theta  \quad \\( \leq   \psi _{\theta} (1,z) -C \log \theta) \,,
\end{multline*}
which provides the second set of assertions of the lemma.
Finally, we observe with the concavity of $ \psi$ (see~\eqref{heatcapacity}) that 
\begin{align*}
e( \theta , z) = \psi( \theta , z) - \theta \psi_\theta ( \theta , z) = \psi ( \theta , z) + ( 0-\theta ) \psi_\theta ( \theta , z) \geq \psi( 0 ,z) \geq 0\,.
\end{align*}
\end{proof}
\section{Existence}
In this section, we prove the existence result of Theorem~\ref{thm:weakin}.
 \subsection{Existence of solutions to the approximate system\label{sec:ex}}
Following a standard approach, we discretize and regularize the state equations in order to find a solution to this approximate system. We use a Galerkin approximation scheme for Maxwell's equation, the energy balance and the phase equation are regularized appropriately.
The standard procedure executed in the sequel of this section consists of proving the local existence of solutions to the approximate problem via Schauder's fixed point argument, deducing \textit{a priori} estimates for the system,  assuring the existence of global solutions to the approximate system, and extracting a converging subsequence, which allows to go to the limit in the approximate system and attain  weak entropy  solution in the limit. 

Consider the following  system, which is discretized in Maxwell's equation and regularized in the energy balance as well as the phase equation. The system is given by
 \begin{subequations}\label{eq:reg}
\begin{align}
\t e^\varepsilon ( \theta , z)   
-\di \left ( \kappa_\varepsilon (\theta , z) \nabla \theta \right ) 
- \frac{\varepsilon}{\theta^2}={}& \sigma^\varepsilon ( \theta ) | \t \f A   |^2  + \delta^{3/2} \tau^\varepsilon (\theta ) | \Delta z|^2    \,,\quad \theta (0) = \theta _0 ^\varepsilon \,, \label{tempreg}\\
\int_0^T \int_D \sigma^\varepsilon (\theta)  \t \f A \cdot \f w  +  \frac{1}{\mu(z)} \curl \f A \cdot \curl \f w \de x\de t ={}& \int_0^T \int_D \f J_s \cdot \f w \de \f x \de t  \,,\quad \f A(0) = P_n \f A_0\,,  \label{magneticreg}\\
 \t z- \delta \Delta z  + \frac{\partial _z \psi^\varepsilon (\theta , z )}{\tau^\varepsilon(\theta)} = {}& 0\,,\quad  z(0) = z_0^\varepsilon  \label{innerreg}
\end{align}
\end{subequations}
for all $ \f w \in Y_n$, where $ \kappa_\varepsilon(\theta , z) = \kappa^\varepsilon ( r, z )+ \varepsilon r^2 $.
By $ \theta_0^\varepsilon$ and  $z^\varepsilon_0$, we denote suitable regularizations of $ \theta_0$ and  $ z_0$, respectively. 
Both equations are equipped with Neumann boundary conditions, i.e.,
\begin{align*}
\f n \cdot \kappa_\varepsilon ( \theta,z) \nabla \theta = 0 \,, \quad \f n \cdot \nabla z  =0  \quad \text{on }\partial \Omega\,.
\end{align*}
The functions $ \psi^\varepsilon $, $\kappa^\varepsilon$, $\sigma^\varepsilon$, and  $ \tau^\varepsilon$ are suitable regularizations of $ \psi$, $\kappa$, $\sigma$, and $\tau$, respectively. This regularizations are chosen in a way that
\begin{align}\label{continuousconvergence}
\psi ^\varepsilon \ra \psi \,, \quad  \text{in }\C^{1}((0,\infty)\times \R; \R) \,, \qquad  \tilde{\kappa}^\varepsilon \ra \tilde{\kappa} \,,\quad  \bar{\kappa}^\varepsilon \ra \bar{\kappa} \,, \quad \sigma^\varepsilon \ra\sigma \,, \quad \tau^\varepsilon \ra \tau\quad  \text{in } \C((0,\infty)\times \R; \R),
\end{align}
respectively and such that the Assumption~\eqref{Asume:1}
are fulfilled independently of $\varepsilon$. 
 Everywhere in $D/(\Omega\cap \Sigma )$, where $\sigma$ is vanishing, we set $\sigma_\varepsilon := \varepsilon $. 

The spaces $\{Y_n\}$ are suitable Galerkin spaces with $ Y_n \subset Y_{n+1} \subset \ldots \subset \f H_{\cur}$ and $P_n$ defines the $L^2$-orthonormal projection onto $Y_n$. 
For example this spaces can be chosen to be spanned by eigenfunctions of the operator $ \curl\curl $ equipped with tangential Dirichlet boundary conditions ($\f n\times \f A = 0$, i.e., on $\f H_{\cur}$), since this is the inverse operator of a self-adjoint compact operator (see~\cite[Section~4.5]{regMax} or~\cite{amrouche} for regularity results concerning this operator). 

We introduce the additional term $\varepsilon \theta^p$ to the heat-conductivity to infer additional regularity for the approximate system. The last term on the right-hand side of~\eqref{tempreg} assures that the solution to the approximate system remains bounded away from zero. The regularization in equation~\eqref{innerreg} is added to retain more regularity on the approximate level, and the last term on the right-hand side of~\eqref{tempreg} is introduced in order to be able to handle the influence of the regularization in~\eqref{innerreg} in the entropy production rate~\eqref{entropy}.
We remark that the $\delta$-regularization is only needed due to the $z$-dependence in $\kappa$. For $\partial_z \kappa\equiv 0$, we could chose~\eqref{eq:reg} with $\delta=0$ as approximate system and thus, the limit in the $\delta$-regularization could be omitted and the proof would simplify considerably.

We collect different  tools needed for the existence proof.
During this step, we omit the regularization parameter $\varepsilon$ for the readability. 
First, we focus on Maxwell's and the phase equation.
\begin{proposition}\label{prop:z}
Let $\Omega $ be of class $\C^2$ and $z_0^\varepsilon\in\C^2(\ov \Omega)$. Furthermore, let $ \theta\in \C (\ov \Omega \times [0,T])$. Then there exists a unique  solution $ z \in \C ( [0,T];\C^2(\ov\Omega))\cap \C^1([0,T]; \C(\ov \Omega))$ of~\eqref{innerreg}.
\end{proposition}
\begin{proof}
This result is an consequence of the fact that the fundamental solution associated to the heat equations forms a strongly continuous semi group on the continuous functions.
%
%
 The additional regularity can be read from the Duhamel formula and the regularity of the data.  
See for instance Taylor~\cite[Prop.~15.3.1 and~15.3.2 together with Ex.~15.3.1]{taylor}.

\end{proof}

\begin{proposition}\label{prop:A}
Let $ \theta\in \C (\ov \Omega \times [0,T])$ and  $ z \in \C(\ov \Omega\times [0,T])$.Then there exists a unique solution $ \f A_n\in \C^1([0,T_n] ; Y_n)$ to the approximate Maxwell equation~\eqref{magneticreg}. 
\end{proposition}
\begin{proof}
Following the standard approach of a Galerkin discretization, we can express the equation~\eqref{magneticreg} as a system of ordinary differential equations. 
It is standard to prove the existence locally in time of solutions to the approximate problem in the sense of Carathe\'{e}odory, i.e., of solutions that are absolutely continuous with respect to time (see, e.g., \cite[Chapter I, Thm.~5.2]{hale}).
%
\end{proof}


In the following, we collect different results for the approximate energy balance. These are very similar to the ones in~\cite[Sec.~3.4.2]{singular}, but generalized to the current case (especially the $z$-dependence in of the heat conduction $\kappa$). First, a  criterion is given, which allows to show the positivity and boundedness of $\theta$ and thus, to test with $1/\theta$ on the approximate level.  
 \begin{lemma}\label{lem:comp}
 Let $ \Omega $ be of class $\C^2$, $z\in \C^1([0,T];\C(\ov \Omega)) \cap \C([0,T];\C^2(\ov \Omega ))$, and  $ \f A \in \C^1([0,T]; Y_n)$, assume that $ \underline{\theta}$ and $ \overline{\theta} $ are sub and super solutions of~\eqref{tempreg}, i.e., for $ \underline{\theta}$ and $\overline{\theta}$~\eqref{tempreg} holds with $=$ replaced by $\leq$ and $\geq$, respectively,     fulfilling 
 \begin{align}
 \begin{split}
 & \underline{\theta },\, \overline{\theta} \in L^\infty(0,T; W^{1,2}( \Omega)) \,, \quad \t \underline{\theta},\,\t \overline{\theta} \in L^2(0,T; L^2(\Omega)) \\
 &\di \left (  \bar{\kappa}(z) \nabla \int_0^{\ov\theta} \tilde{\kappa}(r) \de r \right )   , \, \di \left (  \bar{\kappa}(z) \nabla \int_0^{\un\theta} \tilde{\kappa}(r) \de r \right ) \in L^2(0,T;L^2(\Omega))\\
 &0 < \essinf_{(\f x,t) \in \Omega \times (0,T) } \underline{\theta}\leq \esssup_{(\f x,t) \in \Omega \times (0,T) } \underline{\theta} < \infty\\
 &0 < \essinf_{(\f x,t) \in \Omega \times (0,T) } \overline{\theta}\leq \esssup_{(\f x,t) \in \Omega \times (0,T) } \overline{\theta} < \infty
  \end{split}\label{assumptioncompare}
  \intertext{as well as}
 & \underline{ \theta}(0,\f x ) \leq \overline{\theta}(0,\f x ) \quad \text{a.e.~in }\Omega \notag\,.
\intertext{Then}
&\underline{ \theta}(t,\f x ) \leq \overline{\theta}(t,\f x ) \quad \text{a.e.~in }\Omega \times (0,T)\,.\notag
\end{align}  
 \end{lemma}
 \begin{proof}
This Lemma is a simple adaptation of~\cite[Lemma~3.2]{singular} to the considered case.   
An additional difficulty is that $\kappa
$ depends on $z$. 
  We can multiply the approximate energy equation with $ \sign (  \underline{\theta}-\overline{\theta}) $, where $ \sign (y) = 0 $ if $ y\leq 0$ and $\sign (y) = 1$ if $y>0$. 
  Since $ e(\cdot,z)$ and $ \int^{\cdot}_{0} \tilde{\kappa}(r) \de r $ are monotonic functions, we may infer
  \begin{multline*}
  \t (  e ( \underline{\theta}, z)-e(\overline \theta , z)  ) \sign( e( \underline{\theta},z)-e( \ov \theta ,z) )- 
   \di \left (  \bar{\kappa}(z) \nabla \int_{\ov\theta}^{\un\theta} \tilde{\kappa}(r) \de r \right ) 
   \sign  \left (  \int_{\ov\theta}^{\un\theta} \tilde{\kappa}(r) \de r
   \right ) 
\\  
%
-\varepsilon\left (   \frac{1}{\underline{\theta}^2}-\frac{1}{\ov \theta ^2}\right ) \sign (  \underline{\theta}-\ov \theta ) 
\\  \leq  |  \t \f A |^2   |   \sigma ( \underline{\theta} )-\sigma ( \ov \theta ) |  
\sign ( \underline{\theta}- \ov \theta ) +\delta^{3/2} | \Delta z|^2 | \tau(\underline{\theta})-\tau(\overline{\theta })| \sign (  \underline{\theta}-\ov \theta )
\,.
  \end{multline*}
  We observe that for $ | y|^+:= \max \{ y, 0\} $ it holds  
  $$  |y|^+ =  y \,\sign( y)\,, \quad \t | y| ^+= \t y\, \sign( y) \,, \quad \nabla | y|^+  = \nabla y \, \sign (y) \quad \text{ a.e.~in }\Omega \times (0,T)\,.$$
  Additionally, the monotonicity of $\theta \mapsto  - 1 / \theta^2 $ implies
$ \varepsilon\left (   {1}/{\underline{\theta}^2}-{1}/{\ov \theta ^2}\right )\sign ( \underline{\theta}- \ov \theta ) \leq 0$.
Since, $\sigma
$ and $\tau
$ are Lipschitz continuous and $ \t \f A$ as well as $\Delta z$ are bounded, we infer that 
$$ \left ( | \t  \f A |^2   |  \sigma ( \underline{\theta} )- \sigma ( \ov \theta ) |   
+ \delta^{3/2} | \Delta z|^2 
| \tau (\un \theta ) - \tau(\ov\theta)| \right )\sign( \underline{\theta}- \overline{\theta} )
\leq c |  e( \underline{\theta}, z)-e(\overline{\theta}, z) |^+\,.$$
%
%
%
%
Integrating over $\Omega$, implies that
\begin{multline*}
 \t\int_\Omega |  e ( \underline{\theta}, z)-e(\overline \theta , z)  |^+\de \f x - \int_\Omega  
   \di \left (  \bar{\kappa}(z) \nabla \int_{\ov\theta}^{\un\theta} \tilde{\kappa}(r) \de r \right )  \sign  \left (  \int_{\ov\theta}^{\un\theta} \tilde{\kappa}(r) \de r
   \right )   \de \f x\\ \leq c  \int_\Omega  | e ( \underline{\theta}, z)- e(\overline \theta , z)  |^+\de \f x  \,.
\end{multline*}

Via a suitable approximation of the $\sign$ function and an integration-by-parts, we may observe in the limit of the approximation that
   \begin{align*}
-  \int_{\Omega } \di (\bar{\kappa}(z) \nabla y ) \sign y  \de \f x\geq 0 
    \end{align*}
  for $ y \in W^{2,2}( \Omega) $ with $ \f n \cdot \nabla y =0 $ on $\partial \Omega$ (compare to~\cite[Lemma~3.2]{singular}). This can be shown by approximating the $\sign$-function appropriately and applying an integration-by-parts on the approximate lemma, estimate, and go to the limit afterwards.
  Gronwall's lemma implies with the monotonicity of $e$ in $\theta$ the assertion. 
  
   \end{proof}
 
 \begin{proposition}\label{prop:t}
 Let $ \f A \in \C^1([0,T]; Y_n)$ and  $ z \in \C([0,T];\C^2(\ov\Omega))\cap \C^1([0,T]; \C(\ov\Omega))$ be given and let $ \theta_0^\varepsilon \in \C^2 ( \overline \Omega)$. Then there exists a unique strong solution  $\theta \in \C^{0,\alpha}(\ov \Omega \times [0,T])$ for an $\alpha>0$ to equation~\eqref{tempreg} satisfying the estimate
 \begin{multline*}
\| \theta \|_{L^\infty( W^{1,2})}^2 + \| \t \theta \|_{L^2(0,T; L^2)} ^2 + \| \di ({\kappa} ( \theta,z )\nabla \theta)  \|_{L^2(0,T;L^2)}+ \| \theta \|_{\C^{0,\alpha}(\ov\Omega\times [0,T]) } \\
\leq  g( \| \f A \|_{ C^1([0,T];Y_n)} , \| z\|_{\C1( [0,T];\C(\ov\Omega))\cap\C([0,T];\C^2(\ov\Omega))}+ \| \theta^\varepsilon _0\|_{ W^{1,2}(\Omega)} ) \,,
 \end{multline*}
 where $g$ is bounded function on bounded sets  and  $ \hat{\kappa}
 (\theta  , z ) = \int_0^\theta  \kappa ( r, z ) \de r$. 

 \end{proposition}
\begin{proof}
In a first step, we observe that a solution $\theta$ to~\eqref{tempreg} is unique due to Lemma~\ref{lem:comp}. Additionally, we observe that a constant function is a subsolution to~\eqref{tempreg} as soon as 
\begin{multline*}
\frac{\varepsilon}{\underline\theta^2}  \geq \psi_z(\underline{\theta} , z) \t z -\underline{ \theta} \psi_{\theta z} ( \underline{\theta }, z) \t z +  \int_0^{\underline{\theta}} \partial_z \kappa
(r,z) \de r \Delta z+  \int_0^{\underline{\theta}} \partial_{zz} \kappa%
(r,z) \de r | \nabla z|^2 \\- \sigma(\underline{\theta}) | \t \f A |^2 - \varepsilon ^{3/2}\tau (\un \theta ) | \Delta z|^2 \,.
\end{multline*}
Since all terms on the right-hand side are bounded, for bounded $\underline{\theta}$, we can find a small enough constant $\underline{\theta}>0$ such that the above inequality is fulfilled. 

Reformulating equation~\eqref{tempreg} for a spatially constant function, we observe that $ \overline{\theta}$ is a supersolution to~\eqref{tempreg}, if
\begin{multline*}
\t \overline{\theta} \geq \frac{1}{\ov \theta \psi_{\theta\theta}(\overline{\theta },z)}\left (   \psi_z(\overline{\theta} , z) \t z -\overline{ \theta} \psi_{\theta z} ( \overline{\theta }, z) \t z +  \int_0^{\overline{\theta}} \partial_z \kappa
(r,z) \de r \Delta z+  \int_0^{\overline{\theta}} \partial_{zz} \kappa
(r,z) \de r | \nabla z|^2 \right ) 
\\
+ \frac{1}{-\ov \theta \psi_{\theta\theta}(\overline{\theta },z)}\left (   \sigma(\overline{\theta}) | \t \f A|^2 +\varepsilon ^{3/2}\tau (\ov \theta ) | \Delta z|^2+ \frac{\varepsilon}{\overline{\theta}}\right )
\,.
\end{multline*}
Due to the assumptions, the right-hand side of the previous inequality is bounded from below by~$-c(\ov \theta +1)$. 
We may infer
that $ \ov \theta (t) = ( \sup _{\f x \in \Omega} \theta^\varepsilon_0 (\f x)+1  ) e^{c t}$ is a super solution to~\eqref{tempreg}. We conclude that a solution $\theta$ to~\eqref{tempreg} is essentially bounded, i.e., $\theta \in L^\infty( 0,T ;L^\infty)$  with $\esssup \theta (\f x ,t) \leq  ( \sup _{\f x \in \Omega} \theta^\varepsilon_0 (\f x) +1 ) e^{c T}$.

We are going to prove the estimates. The existence is then a routine matter of finding a suitable regularization or discretization and passing to the limit (see~\cite[Sec.~3.4.3]{singular}).
First, we observe that $ \t e( \theta, z) = - \psi_{\theta\theta}( \theta , z)\theta\t  \theta  + \psi_{z} ( \theta ,z) \t z - \theta \psi_{ \theta z} ( \theta , z) \t z \,.$   
Testing~\eqref{tempreg} by $\theta $ implies 
\begin{multline*}
\t\int_{\Omega} \int_0^\theta (-\psi_{\theta\theta} ( r ,z) r^2 )\de r\de \f x+ \int_\Omega \kappa
( \theta ,z  ) | \nabla \theta|^2 \de \f x \\ \leq \int_\Omega\sigma(\theta ) \theta |\partial _t  \f A|^2+ \delta^{3/2} \tau{(\theta)} \theta | \Delta z|^2    - \psi_z ( \theta , z) \t z \theta 
+ \theta^2 \psi_{\theta z}(\theta,z) \t z + \int_0^\theta \psi_{\theta\theta z}( r,z)r^2 \de r \t z \de \f x
\end{multline*}
Due to the regularity of $\f A$ and $z$ as well as the growth conditions~\eqref{heatcapacity}, the right-hand side can be estimated by $ c( | \theta|^2+1)$. 
Integrating in time, Using again~\eqref{heatcapacity} to estimate the left hand side from below and applying Gronwall's estimate, we may infer
\begin{align}
\| \theta\|_{L^\infty(L^2)}^2 + \| \theta\|_{L^2(W^{1,2})}^2 \leq c \,.\label{disestone}
\end{align}
Testing now equation~\eqref{tempreg} by $ \t \hat{\kappa}
(\theta , z )= \kappa_\varepsilon(\theta ,z) \t \theta + \int_0^\theta \partial _z \kappa
(r, z) \de r \t z $ implies
\begin{multline*}
\int_\Omega \left (-  \theta\psi_{\theta,\theta} ( \theta,z) \kappa
(\theta,z) \right )| \t \theta|^2 \de \f x + \t \int_\Omega \frac{1}{2}| \nabla \hat{\kappa}
( \theta,z)|^2 \de \f x 
=\\ \int_{\Omega}  \t \hat \kappa
(\theta ,z) \left (  \sigma(\theta) | \t \f A |^2 + \delta^{3/2} \tau(\theta ) | \Delta z|^2 + \frac{\varepsilon}{\theta^2} \right ) + \theta \psi_{\theta,\theta}(\theta ,z) \t \theta \int_0^\theta \partial_z \kappa
(r,z)\de r \t z \de \f x
\\
- \int_{\Omega}  \t \hat \kappa
(\theta ,z) \left (  \di \left (\int_0^\theta \partial_z \kappa
( r, z) \de r \nabla z \right )  + \psi_z ( \theta,z) \t z - \theta \psi_{\theta z}( \theta,z) \t z \right ) \de \f x
\,.
\end{multline*}

The right-hand side can be further estimated. Therefore, we exploit~\eqref{heatcapacity}, 
\begin{align*}
 \t \hat{\kappa}
(\theta , z )={}& \kappa
(\theta ,z) \t \theta + \int_0^\theta \partial _z \kappa
(r, z) \de r \t z \, \intertext{and} 
 \di \left (\int_0^\theta \partial_z \kappa
( r, z) \de r \nabla z \right ) ={}&\partial_z \kappa
(\theta ,z) \nabla \theta \cdot\nabla z + \int_0^\theta \partial_z \kappa
( r, z) \de r \Delta z + \int_0^\theta \partial_{zz} \kappa
( r, z) \de r | \nabla z|^2 \,, \end{align*}

to estimate 
\begin{multline*}
\frac{1}{2}\int_\Omega \left (-  \theta\psi_{\theta,\theta} ( \theta,z) \kappa
(\theta,z) \right )| \t \theta|^2 \de \f x + \t \int_\Omega \frac{1}{2}| \nabla \hat{\kappa}
( \theta,z)|^2 \de \f x 
\\ \leq  c \int_\Omega \kappa
( \theta ,z) \left ( ( \partial_z \kappa
( \theta,z) )^2 | \nabla z|^2 | \nabla\theta|^2+\left ( \int_0^\theta \partial_{zz} \kappa
(r,z) \de r\right )^2 | \nabla z|^4 +\left ( \int_0^\theta \partial_z \kappa
(r,z) \de r\right )^2| \Delta z|^2  \right )  \de \f x \\
+c \int_\Omega | \t z|^2\left (  \kappa
(\theta,z) + 1 + \left ( \int_0^\theta \partial_z \kappa
(r,z) \de r \right )^2 \right ) + \kappa
(\theta ,z) \left ( \frac{\varepsilon^2}{\theta^4 } + \sigma^2(\theta ) | \t \f A |^4 + \delta ^3 \tau^2(\theta) | \Delta z|^4   \right )  \de \f  x \,,
\end{multline*}
where all terms depending on $\t \theta$ were estimated by Young's inequality such that all contributions of $\t \theta$ can be absorbed in the left-hand side. 
Further on, we estimate
\begin{multline*}
\underline{\kappa}\kappa
(\theta ,z) | \nabla \theta |^2 \leq  \kappa^2
(\theta,z) | \nabla \theta |^2 \leq \frac{1}{2} | \nabla \hat\kappa
( \theta,z) |^2 + c \left |\int_0^\theta \partial_z \kappa
(r, z) \de r  \right |^2 | \nabla z |^2
\\
\leq \frac{1}{2} | \nabla \hat\kappa
( \theta,z) |^2 + c \left |\theta  \right |^2 | \nabla z |^2\,.
\end{multline*}
In view of the estimate~\eqref{disestone}, the last term on the right hand side is bounded. Note that $ \kappa
(\theta ,z) \leq c ( \theta^2 +1) $ and $\partial_z \kappa
\leq  C$. 
Note that the derivatives of $\kappa
$ can be assumed to be bounded due to the regularization in this stage.


Combining the last two estimates together with the regularity of $\f A$ and $z$ as well as the growth conditions~\eqref{heatcapacity}, we can apply Gronwall's lemma to observe
\begin{align*}
\| \theta \|_{L^\infty(W^{1,2})} ^2 + \| \nabla \kappa
(\theta,z) \|_{L^\infty(L^2)}+ \| \theta \|_{W^{1,2}(L^2)}^2 \leq c \,.
\end{align*} 
Comparison in equation~\eqref{tempreg} allows to infer the last  bound on~$\| \di ( \kappa
(\theta, z)\nabla \theta )\|_{L^2(L^2)} $. 
We can reformulate the equation as a parabolic equation for $e$. 
Indeed, observing that $ \nabla \theta = ( \nabla e - \partial_z \psi \nabla z + \theta \partial_{\theta z} \psi \nabla z ) / ( - \theta \partial_{\theta \theta } \psi ) $, we find 
\begin{multline*}
\t e - \di ( \kappa^*
( e,z) \nabla e) =\\  \varepsilon \frac{1}{\theta } + \sigma(\theta ) | \t \f A |^2 +\delta ^{3/2} \tau( \theta )| \Delta z|^2   -\di \left (\kappa
(\theta , z ) \frac{ \partial_z \psi(\theta ,z) \nabla z - \theta \partial_{\theta z} \psi(\theta ,z) \nabla z}{- \theta \partial_{\theta \theta } \psi(\theta ,z)}\right )  \,,
\end{multline*}
where $ \kappa^*$ is given by $ \kappa^*
(e ,z) = \kappa
(\hat{e}(e,z) , z )/ ( - \hat{e}(e,z) \partial_{\theta\theta} \psi(\hat{e}(e,z) ,z) $ (see Lemma~\ref{lem:thetae} for the definition of $\hat e$). 
Due to the $L^\infty$-bound on $\theta$ and Lemma~\ref{lem:thetae}, we  observe that $e$ is bounded, i.e., $e\in L^\infty(\Omega \times (0,T))$, a standard regularity result for parabolic quasilinear partial differential equations (see~\cite[Corollary~4.2(ii)]{trudinger}) provides that $e \in \C^{0,\alpha}(\ov \Omega \times [0,T])$ for some $\alpha>0$.  Note that the right-hand side of the previous equation is bounded in some suitable $L^p$-space.
By Interpolation we also observe from $z\in\C([0,T];\C^2(\Omega))\cap \C^1([0,T];\C(\Omega))$ that $ z \in C^{0,\alpha} (\ov \Omega \times [0,T])$ (see Lunardi~\cite[Example~1.1.8]{lunardi}). 
Since the inverse mapping $ (e,z) \mapsto \theta $ is even a $\C^1$-mapping (see Lemma~\ref{lem:thetae}), we find that also $ \theta \in \C^{0,\alpha} (\ov\Omega \times [0,T])$. 

\end{proof}

In the following, 
we argue by Schauder's fixed point theorem that there exists at least locally a solution to the regularized discretized system~\eqref{eq:reg}.
The proof is similar to the one in~\cite[Section~4.1]{elisabetta}, but on the space of continuous functions.
\begin{proposition}
There exists a solution to the system~\eqref{eq:reg} on $[0,T_n]$ for a fixed $T_n$.
\end{proposition}
\begin{proof}
Consider a small enough fixed  $T_n\in (0,T]$ (which will be specified later) and a fixed big enough constant $R> 2 \sup _ { \f x\in \Omega} \theta_0^\varepsilon$ and introduce the space 
$$ \mathbb{B}_R:= \{ \theta \in \C(\ov \Omega \times [0,T])  : \| \theta \|_{ \C(\ov \Omega \times [0,T]) }\leq R\}\,.$$
The fixed point map $\mathbb{S} :  \mathbb{B}_R \ra  \mathbb{B}_R$ is defined via: For given $\tet$ solve~\eqref{innerreg} according to~Proposition~\ref{prop:z} to get the solution~$z_{\tet}$, afterwards solve~\eqref{magneticreg} according to Proposition~\ref{prop:A} to get the solution~$\f A_{\tet}$ in the last step solve~\eqref{tempreg} according to~Proposition~\ref{prop:t}. 
The continuity of $\mathbb{S}$ follows directly from the three propositions. Since $ \C^{0,\alpha} (\ov\Omega \times [0,T])$ is compactly embedded into $\C(\ov\Omega \times [0,T])$, $\mathbb{S}$ is even compact. 
That $\mathbb{S}$ is well-defined on $ \mathbb{B}_R$, can be observed by the inequality $ \|\theta \|_{ \C(\ov \Omega \times [0,t]) } \leq t^\alpha \|  \theta \|_{ \C^{0,\alpha} (\ov \Omega \times [0,t]) }  $ for small enough $t$. 
 Schauder's fixed point theorem grants that $\mathbb{S} $ admits at least one solution. 

\end{proof}
In order to extend the solution to the whole time interval $[0,T]$, we need to prove suitable global \textit{a priori} estimates. This is done in the next section.


\subsection{\textit{A priori} estimates independent of first regularization\label{sec:firstapri}}
In the first step we derive \textit{a priori} estimates independent of $\delta$.  To remain the lucidity, we omit the dependence of the solutions on the regularization and discretization parameter $\delta$, $n$, and $\varepsilon$ in this step.

Note that in this step, the space of the Galerkin discretization of the Maxwell equation remains the same, i.e., $n $ is fixed.
In this case~\eqref{magneticreg} is just a linear ordinary differential equations in $\R^n$ for $\f A$. Since $\mu$ and $\sigma$ are bounded independently of $z$ and $\theta$, respectively, we may infer that 
\begin{align}
 \| \t \f A \|_{L^2(\R^n)} \leq c\,.\label{Aestdis}
\end{align}


Mimicking the energy estimate of the system, we test equation~\eqref{tempreg} with~$1$,~\eqref{magneticreg} with $\t \f A $ and add them up,
\begin{multline}
\t \int_\Omega e ( \theta _ \varepsilon, z _ \varepsilon ) \de \f x  =\\ \int_\Omega \frac{\varepsilon}{\theta^2_\varepsilon} + \delta ^{3/2} \tau(\theta ) | \Delta z|^2\de \f x - \int_ \Omega \frac{1}{\mu(z_\varepsilon)}  \curl \f A\cdot  \curl \t \f A  \de \f x
- \int_{D/\Omega} \sigma_{\out} | \t \f A|^2 \de \f x + \int_\Sigma \f J_s \cdot \t \f A \de \f x 
\,.\label{energyregdis}
\end{multline}
Note that $ \int_\Omega \di (( \kappa ( \theta_\varepsilon,z_\varepsilon ) + \varepsilon \theta^2 )\nabla \theta_\varepsilon ) \de \f x = \int_{\partial\Omega} \f n \cdot ( \kappa ( \theta_\varepsilon,z_\varepsilon ) + \varepsilon \theta^2 )\nabla \theta_\varepsilon  \de S =0 $ such that the contribution due to the second term in equation~\eqref{tempreg} vanishes.

From the definition of the energy function $e$ and~\eqref{innerreg}, we observe
\begin{align*}
\t e ( \theta ,z) = \psi_z(\theta ,z) \t z - \theta \t \psi_{\theta} ( \theta ,z) = - \theta \t \psi_{\theta} ( \theta ,z) - \tau(\theta ) | \t z|^2 + \delta \tau(\theta) \t z \Delta z \,.
\end{align*}
Testing now equations~\eqref{tempreg} with $-1/\theta $ and integrating-by-parts, we may infer
\begin{multline}
\t \int_\Omega \psi_{\theta} ( \theta ,z) \de \f x + \int_\Omega \kappa_\varepsilon ( \theta ,z) | \nabla \log \theta |^2 + \frac{\varepsilon}{\theta^3} + \sigma ( \theta ) \frac{| \t \f A |^2}{\theta }+ \delta ^{3/2} \tau(\theta ) \frac{| \Delta z|^2 }{\theta} +\tau(\theta) \frac{| \t z|^2}{\theta}\de \f x \\
=   \int_\Omega  \tau(\theta ) \frac{\t z }{\sqrt \theta } \frac{\delta \Delta z}{\sqrt{\theta}} \de \f x .\label{entropydisreg}
\end{multline}
Testing~\eqref{innerreg} with $\t z $, integrating by parts and estimating by Young's inequality implies
\begin{align}
\frac{1}{2}\| \t z \|_{ L^2(L^2)} ^2 + \delta \| \nabla z\|_{L^\infty(L^2)} ^2  \leq  \int_0^T \int_\Omega|\psi_z(\theta ,z) |^2 \de \f x \de t \leq C \,.\label{tzdis}
\end{align}
Taking the gradient with respect to the spatial variables  in~\eqref{innerreg} implies
\begin{align*}
\t \nabla z -\delta \Delta \nabla z + \frac{\theta \tau'(\theta)\nabla \log \theta }{\tau^2(\theta )} \psi_z (\theta ,z) -\frac{1}{\tau (\theta)} \theta \psi_{z\theta} (\theta , z) \nabla \log \theta - \frac{1}{\tau(\theta)} \psi_{zz}(\theta , z) \nabla z \,,
\end{align*}
such that testing with $\nabla z$, we find with Gronwall's estimate and Assumption~\eqref{Asume:1} that 
\begin{align}
\| z\|_{L^\infty(W^{1,2})} ^2 + \delta \| \Delta z\|_{L^2(L^2)}^2 \leq c ( \| z_0 \|_{W^{1,2}} + \| \nabla  \log \theta \|_{L^2(L^2)}^2)\,.\label{zesti}
\end{align}

Choosing now a constant $ \gamma$ big enough, we multiply~\eqref{energyregdis} by $\gamma$ add~\eqref{entropydisreg}, add additionally~\eqref{zesti} multiplied by $1/\gamma $ and integrate in time.
Then every term on the right-hand side of the resulting equation has to be controlled by a term on the left-hand side. 
All terms on right-hand side of~\eqref{energyregdis} depending on $\t \f A$ are bounded due to~\eqref{estimate},~\eqref{zesti}, and~\eqref{Aestdis}. We remark that all norms of $\t \f A$ are equivalent in this stadium of the Galerkin approximation. 
The first term on the right-hand side of~\eqref{energyregdis} can be absorbed on the left-hand side of~\eqref{entropydisreg} by Young's inequality
$\varepsilon \theta^{-2} \leq \varepsilon  \theta^{-3}/(2\gamma ) + \varepsilon C $. \label{page:est}

Young's inequality allows to estimate the right-hand side of~\eqref{entropydisreg} and absorb both terms  in the left hand side of the same equation. Note that $ \delta^2 \leq \delta^{3/2}$ for $\delta \in [0,1]$. 
The second term on the right-hand side of~\eqref{zesti} may be absorbed on the left-hand side of~\eqref{entropydisreg} for $\gamma $ big enough. 
As a last point, we observe that due to Lemma~\ref{lem:thetae} $\psi_\theta (\theta,z) \geq -c ( \log \theta + 1 ) \geq -c (\theta +1 ) $ (for $ \theta \leq 1 $, it holds $ - \log \theta \geq 0$), which may be absorbed (again due to Lemma~\ref{lem:thetae}) by the first term on the left-hand side of~\eqref{energyregdis} due to the internal energy for $\gamma$ big enough. 
Finally, the second term on the right-hand side of~\eqref{energyregdis} can be estimated by the second term on the left-hand side of~\eqref{zesti} for $\delta $ small enough.

We may observe from the left-hand side of~\eqref{energyregdis}
\begin{align}
\| e( \theta,z) \|_{L^\infty(L^1)}  \leq c \,\label{edis}
\end{align}
such that due to Lemma~\ref{lem:thetae}, we conclude
\begin{align}
\| \theta \|_{L^\infty(L^1)}  \leq c \,.\label{thetdis}
\end{align}

From the left-hand side of~\eqref{entropydisreg}, we infer the estimates
\begin{multline}
\| \psi_ \theta ( \theta ,z) )  \| _{ L^\infty ( \f L^1)} + \|  \nabla \log \theta  \| _{ L^2 (  L^{2})}^2 + \varepsilon \| \nabla \theta^{p/2}\|_{ L^2(L^2)}^2 + \left \| \frac{\partial_t \f A  }{\sqrt{\theta}}\right \|_{L^2(\f L^2)}^2 + \left \| \frac{\partial _t z}{\sqrt{\theta}}\right \|_{L^2 ( \f L^2)}^2\\+  \varepsilon \left \| \frac{1}{\theta}\right \|_{L^3(L^3)} ^3+\delta^{3/2} \left\|  \frac{\Delta z }{\sqrt \theta }\right \| ^2 _{L^2(L^2)}\leq c \,\label{estimate}
\end{multline}
From Lemma~\ref{lem:thetae} and the regularizing term in $\kappa_\varepsilon$, we may infer that
\begin{align*}
\| \log \theta \|_{L^\infty(L^1)}+ \varepsilon\| \nabla \theta \|_{L^2(L^2)}^2 \leq c \,.
\end{align*}
Note that the estimate of $\nabla \theta$ depends on~$\varepsilon$ in this step. In Section~\ref{sec:weakss} we succeed to deduce a similar estimate independent of $\varepsilon$. 

From the internal energy balance, we observe
by comparison  since $ | \t \f A|^2$ is bounded in $  L^1$  that 
\begin{align*}
\| \t e \|_{ L^1 ( W ^{-1,q} )} \leq c \qquad \text{with } q \in ( 1, 3/2)\,.
\end{align*}

The chain rule implies
\begin{align*}
\n e ( \theta , z) = \n ( \psi( \theta , z) - \theta  \psi_\theta ( \theta ,z) ) = \psi _z ( \theta , z) \n z - \theta  \psi_{\theta ,\theta} ( \theta ,z) \n \theta -\theta \psi_{\theta,z} ( \theta ,z) \n z   \,
\end{align*}
such that we can estimate, using the properties of the free energy~$\psi$ (see~\eqref{heatcapacity}),
\begin{align*}
\|\n e (  \theta ,z) \| _{ L^2(L^2)} \leq c (\|\nabla z \|_{ L^2( L^2)} + \| \nabla \theta\|_{ L^2 ( L^2)} ) \leq c  \,.
\end{align*}

From the entropy estimate~\eqref{estimate}, we find that $\{1/\theta_\delta  \}$ is bounded in $L^3(\Omega \times (0,T))$ and from~\eqref{thetdis} $ \{\theta\} $ is bounded in $ L^\infty(0,T;L^1(\Omega))$. Together with Lemma~\ref{lem:thetae}, we find that $\{ \psi_\theta ( \theta_\delta , z_\delta) \}$ is bounded in any $L^p(\Omega \times (0,T)) $ for $p\in (1,\infty)$. 


 \subsection{Convergence for vanishing first regularization\label{sec:convfirst}}
By standard arguments, we infer the following weak$^*$ convergences
\begin{align}
\theta _\delta & \rightharpoonup \theta \quad \text{in } L^2 ( 0,T ; W^{1,2}(\Omega))\,,  \\
z _ \delta & \stackrel{*}{\rightharpoonup} z \quad \text{in } W^{1,2}( 0,T; L^2(\Omega)) \cap L^\infty (0,T ; W^{1,2}(\Omega))\,,
\\
 \f A _\delta & \rightharpoonup   \f A \quad \text{in } W^{1,2}(0,T;\R^n) \,, \label{w:Adis}
\\
e_\delta & \rightharpoonup e \quad \text{in } L^2(0,T; W^{1,2}(\Omega) )  \,,
\\
\log \theta_\delta & \rightharpoonup  \eta \quad \text{in } L^2 ( 0,T; H^1(\Omega))  \,,
\\
\t e_\delta  & \stackrel{*}{\rightharpoonup} \t e \quad \text{in } \mathcal{M}([0,T];W^{-1,q}(\Omega)) \,,
\\
\log \theta_\delta &\stackrel{*}{\rightharpoonup}  \eta \quad \text{in } L^\infty ( 0,T;\mathcal{M}(\ov\Omega))
\,.
\end{align}
The Lions--Aubin lemma (see~\cite[Cor.~7.9]{roubicek} or~\cite[Thm.~3.22]{BV}) grants that
\begin{align}
z _{\delta} & \ra z \quad \text{in } L^2(0,T;L^2(\Omega)) \,,\label{strongzdis}
\\
\f A_\delta & \ra \f A \quad \text{in } L^2 (0,T ; \R^n) \,,\label{strongAdis}
\\
e_\delta & \ra e \quad \text{in } L^2 ( 0,T ; L^{2}(\Omega))    \,.\label{strongedis}
\end{align}
To infer the strong convergence of $\theta_\delta$, we observe that the mapping $ \theta \mapsto e ( \theta,z) = \psi(\theta,z) - \theta \psi_\theta ( \theta ,z) $ has a continuous bounded inverse (see Lemma~\ref{lem:thetae}).   
Due to~\eqref{strongedis}, we can extract a subsequence that converges a.e. in $\Omega \times (0,T)$ and  a dominating function in $L^2(0,T; L^{2}(\Omega))$  to the sequence $\{ e_\delta\}$ (see~\cite[Theorem~4.9]{brezisbook}). 
With Lemma~\ref{lem:thetae}, we can find a dominating function for the sequence~$\{ \theta_\delta\}$. 
We may express 
$ \theta_\delta =  \hat e( e_\delta, z_\delta) $ (see Lemma~\ref{lem:thetae} for the definition of $\hat e$) such that the point-wise strong convergence of the temperatures follows from the continuity of the inverse function $\hat e$. The dominating function, Lemma~\ref{lem:thetae}, together with Lebesgue's theorem on dominated convergence implies
\begin{align}
\theta_{\delta} =\hat e  (e_ \delta , z_\delta)&\ra \hat e (e,z) = \theta \quad \text{in }L^2(0,T; L^{2}(\Omega)) \,.\label{strongtempdis}
\end{align}

The continuity of the logarithm and $\psi_\theta$
allows to identify
\begin{align}
\log \theta_\delta & \ra  \log \theta  \quad\text{and}\quad \psi_\theta (\theta_\delta , z_\delta ) \ra \psi_\theta ( \theta ,z) \quad  \text{in } L^2 ( 0,T; L^2(\Omega)) \,.\label{stronglogdis}
\end{align}

 Together, we may infer different weak convergences
 \begin{align}
 \begin{split}
 \frac{\sqrt{\sigma(\theta_\delta)}}{\sqrt{\theta_\delta}} \t \f A_\delta & \rightharpoonup  \frac{\sqrt{\sigma(\theta)}}{\sqrt{\theta}} \t \f A  \quad \text{in } L^2  (0,T; \f L^2(\Omega))\,, \\
 \sqrt{\kappa(\theta_\delta , z_\delta)} \nabla \log \theta_\delta & \rightharpoonup   \sqrt{\kappa(\theta , z)} \nabla \log \theta  \quad \text{in } L^2  (0,T; \f L^2(\Omega))\,, \\ 
 \sqrt{\tau(\theta_\delta)} \partial_t z_\delta  & \rightharpoonup   \sqrt{\tau(\theta)} \partial_t z  \quad \text{in } L^2  (0,T; \f L^2(\Omega))\,,
 \end{split}\label{nonlinweakdis}
 \end{align}
where we were able to identify the nonlinear limits  due to the strong convergences~\eqref{strongzdis} and~\eqref{strongtempdis}. 


Testing now~\eqref{tempreg} with~$- \vartheta / \theta_\delta$, where $ \vartheta \in \C^\infty  ( \ov  \Omega \times [0,T])$  with $ \vartheta \geq 0$, we observe that the regularized version of the entropy inequality holds. 
\begin{multline}
\left .\int_\Omega \psi_\theta (\theta_\delta , z_\delta) \vartheta \de \f x \right |_0^t \\+ \int_0^t \int_\Omega \vartheta \left ( \kappa_\varepsilon(\theta_\delta,z_\delta) |\nabla \log \theta _\delta|^2 + \sigma (\theta_\delta ) \frac{ | \t \f A _\delta|^2}{\theta_\delta }+ \tau(\theta _\delta )\frac{|\t z_\delta|^2 }{\theta_\delta } + \frac{\varepsilon}{\theta_\delta  ^3}+ \delta^{3/2}\tau(\theta_\delta)  \frac{|\Delta z_\delta|^2 }{\theta _\delta  }\right )  \de \f x \de t  \\
- \int_0^t \int_\Omega \kappa_\varepsilon(\theta_\delta ,z_\delta) \nabla \log \theta_\delta \nabla \vartheta \de \f x \de t \leq \int_0^t \int_\Omega \psi_\theta ( \theta_\delta ,z_\delta) \t \vartheta  + 
\delta \tau(\theta_\delta ) \frac{\t z_\delta \Delta z _\delta  }{\theta _\delta}\de \f x \de t \,\label{entropiedis}
\end{multline}
for a.\,e. $t\in[0,T]$. 
For the term $ \delta \tau(\theta_\delta )( \t z _\delta\Delta z_\delta  ) / \theta_\delta $, we observe form~\eqref{estimate} that
\begin{align*}
 \delta \int_0^T \int_\Omega \tau(\theta_\delta )\frac{ \t z_\delta \Delta z_\delta  }{ \theta_\delta}\de \f x \de t  \leq{}& \delta ^{1/4} \int_0^T \int_\Omega\tau(\theta_\delta) \left ( \sqrt{\frac{| \t z_\delta|^2 }{\theta_\delta }} \sqrt{\frac{\delta^{3/2}|\Delta z_\delta|^2 }{\theta_\delta }} \right )\de \f x \de t \\ \leq{}& \delta^{1/4} c \left \| \frac{\t z_\delta }{\sqrt{\theta_\delta}}\right \|_{L^2(\Omega\times (0,T))} \left  \| \frac{\delta^{3/4} \Delta z_\delta }{\sqrt{\theta_\delta }}\right \|_{L^2(\Omega\times (0,T))}  \ra 0 \quad \text{as }\delta \ra 0 \,.
\end{align*}
The last $\delta$-dependent term $ \delta^{3/2}\tau(\theta_\delta){|\Delta z_\delta|^2 }/{\theta_\delta }$ may be estimated from below by zero. 
For the remaining terms in the first line of~\eqref{entropiedis} except the first one, converge due to
 the weak-lower semi-continuity of convex functions (\textit{cf}.~\cite[Thm.~10.20]{singular} or~\cite{ioffe}).

To go to the limit in the entropy term $\psi_\theta$ in the first line of~\eqref{entropyregularized}, we first observe that 
$\{ \sqrt{\theta_\delta}\}$ converges strongly in $L^1(\Omega\times (0,T))$ (this follows from Vitali's theorem due to the point-wise strong convergence and weak-compactness in $L^1(\Omega\times (0,T))$ (see~\cite{vitali} and \cite[Thm.~1.4.5]{roubicekmeasure}))
 such that we may find a dominating function
 $h \in L^1(\Omega\times (0,T))$ of this sequence~\cite[Thm.~4.2]{brezisbook}. With Lemma~\eqref{lem:thetae}, we may bound the entropy from below by 
\begin{align*}
\psi_\theta ( \theta_\delta(t),z _\delta  (t)) \geq - c (\log \theta_\delta(t)-1) \geq - c( \sqrt{\theta_\delta (t)}-1) \geq - c (h(t)-1)\,.
\end{align*}
Due to the almost everywhere convergence, we may go to the limit in $ \psi_\theta (\theta_\delta(t) , z_\delta(t))$ by Fatou's lemma (see~\cite{fatou}). Since $ \psi_\theta (\theta_\delta(t)  , z_\delta (t))+c h(t)$  is a positive function, we observe
\begin{align*}
\int_{\Omega} \psi_\theta ( \theta (t) , z (t)) \de \f x \leq \liminf_{\delta\ra 0} \int_\Omega  \psi_\theta ( \theta_\delta (t) , z_\delta  (t)) \de \f x\quad\text{a.e.~in }(0,T)
\,.
\end{align*}

In the limit $\delta \ra 0$, we observe that the regularized version of the entropy inequality holds,
\begin{multline}
\left .\int_\Omega \psi_\theta (\theta , z) \vartheta \de \f x \right |_0^t + \int_0^t \int_\Omega \vartheta \left ( \kappa_\varepsilon(\theta,z) |\nabla \log \theta |^2 + \sigma (\theta ) \frac{ | \t \f A |^2}{\theta }+ \tau(\theta)\frac{|\t z|^2 }{\theta } + \frac{\varepsilon}{\theta^3}\right ) \de \f x \de t  \\
- \int_0^t \int_\Omega \kappa_\varepsilon(\theta ,z) \nabla \log \theta \nabla \vartheta \de \f x \de t \leq \int_0^t \int_\Omega \psi_\theta ( \theta ,z) \t \vartheta \de \f x \de t \,\label{entropiedis2}
\end{multline}
for a.\,e. $t\in(0,T)$. 
The boundedness due to~\eqref{estimate} shows that we can find  a  measure~$\nu\in \M(  \Omega  \times (0,T) )$ such that 
 \begin{align}
\int_0^T \int_\Omega  \t \psi_\theta (\theta ,z)  \vartheta \de \f x \de t   +\int_0^T \int_\Omega \vartheta  \de  \nu
 -\int _0^T \int_\Omega   \kappa _\varepsilon ( \theta ,z) \nabla \log \theta \cdot \nabla \log \vartheta \de \f x \de t   = 0\,\label{measureinterpretation}
 \end{align}
 in the sense of distributions, i.e., for all $\vartheta  \in \C^\infty_c( \Omega \times (0,T)) $, where the inequality 
 \begin{align*}
 \int_0^T \int_\Omega \vartheta \de  \nu \geq  \int_0^T \int_\Omega \vartheta\left (  \kappa_\varepsilon( \theta , z) | \nabla \log \theta |^2 + \sigma (\theta) \frac{| \t \f A |^2}{  \theta} + \tau(\theta ) \frac{| \t z|^2 }{ \theta}+ \frac{\varepsilon}{\theta^3} \right ) \de \f x \de t 
\end{align*}  
  holds for all $ \vartheta \in \C^\infty_c (  \Omega \times (0,T))$  with $ \vartheta \geq 0$. Note that $ \delta^{3/2} \tau(\theta_\delta) | \Delta z_\delta|^2 /\theta_\delta \geq 0 $ for all $ \delta \geq 0$. 

The convergences~\eqref{w:Adis} and~\eqref{strongAdis} allow to go to the limit in the Galerkin approximation of Maxwell's equation
\begin{align}
\int_0^T \int_D \sigma (\theta)  \t \f A\cdot \f w  +  \frac{1}{\mu(z)} \curl \f A \cdot \curl \f w \de x\de t ={}& \int_0^T \int_D \f J_s \cdot \f w \de \f x \de t  \,,\quad \f A(0) = P_n \f A_0\,,  \label{magneticdis}\,
\end{align}
for all $ \f w \in L^2(0,T; Y_n)$. 
Testing now~\eqref{magneticreg} for the limit $\f A$ subtracted from~\eqref{magneticreg}  for $\f A_\delta$  with $ ( \t \f A_\delta  - \t \f A ) $, we observe strong convergence of the time derivative  in $L^2$, i.e.,
\begin{align*}
\int_0^T \int_D  \sigma (\theta_\delta  ) | \t \f A_ \delta - \t \f A |^2 \de \f x \de t ={}& - \int_0^T \int_D \left (  \frac{1}{\mu(z_\delta) } \curl \f A _\delta  - \frac{1}{\mu(z)} \curl \f A \right ) \cdot \left ( \curl \t \f A_\delta - \curl  \t \f A \right  )   \de \f x \de t
\\
&- \int_0^T \int_D ( \sigma (\theta _\delta) -\sigma(\theta) )  \t \f A \cdot \left ( \t \f A_\delta - \t \f A\right ) \de \f x \de t
 \,.
\end{align*}
Due to~\eqref{w:Adis},~\eqref{strongtempdis}, and~\eqref{strongAdis} and the equivalence of the norms on the discrete space $Y_n$, the right-hand side converges. From~\eqref{zesti}, we also find $ \delta^{3/2}\int_0^T\int_\Omega\tau(\theta_\delta) | \Delta z_\delta |^2\de \f x \de t = \sqrt{\delta} c ( \delta \| \Delta z_\delta\|_{L^2(\Omega \times (0,T))}^2 )  \ra 0$ as $ \delta \ra 0$. 
Such that  going to the limit in~\eqref{tempreg} (in the weak formulation) implies the equality 
\begin{align}
- \int_0^T\int_\Omega e(\theta,z) \t \zeta  \de \f x \de t + \int_0^T \int _\Omega \kappa_\varepsilon (\theta ,z) \nabla \theta \cdot \nabla \zeta \de \f x \de t  =\int_0^T \int_\Omega \sigma (\theta) | \t \f A|^2 \zeta + \frac{\varepsilon}{\theta ^2}\zeta \de \f x\de t \,\label{energyreg}
\end{align}
for all $ \zeta \in \C_c^\infty(\Omega \times (0,T))$. 
Note that $ 1 / \theta^2 _\delta \ra 1 / \theta^2 $ in $L^1$ by Vitali's theorem (cf.~\cite{vitali})  since the sequence is relatively weakly compact~\cite[1.4.5]{roubicekmeasure}%
~(see~\eqref{estimate}) and $ \theta_\delta $ converges strongly point wise in $\Omega \times (0,T)$ (see~\eqref{strongtempdis}). 

Due to~\eqref{zesti}, we infer
\begin{align*}
\int_0^T \int_\Omega \sqrt\delta \tau(\theta_\delta) \delta   \sqrt{\delta}\Delta z_\delta    y \de \f x \de t \leq c \sqrt\delta \left ( \sqrt\delta  \| \Delta z_\delta \|_{L^2(L^2)}\right ) \| y \| _{L^2(L^2)} \ra 0 \quad \text{as }\delta \ra 0
\end{align*}
Such that the limit of~\eqref{innerreg} as $\delta \ra 0$ is given by 
\begin{align}
\tau(\theta) \t z  + {\partial _z \psi (\theta , z )} = {}& 0\text{ a.e.~in }\Omega \times (0,T)\quad
 z(0) = z_0^\varepsilon  \label{innerdis}\,.
\end{align}
The solution $( \theta_{n\varepsilon} , z_{n\varepsilon} , \f A_{n\varepsilon})$ fulfills in this step the system~\eqref{entropiedis}--\eqref{innerdis}.

\subsection{\textit{A priori} estimates independent of Galerkin approximation and second regularization\label{sec:apri2}}
First, we have to perform the limit procedure in $n$, the index of the Galerkin discretization, and then in $\varepsilon$, the regularization parameter. Since both limit procedures are very similar, we perform the essential estimates and convergence arguments only ones. 
The difference to the previous step is, that we have now the $L^\infty$-estimate of the time derivative of $z$ at our disposal, i.e., comparison in~\eqref{innerdis} grants that
(compare to~\cite[Lemma~2.5]{surfacehardening})
\begin{align}
\| z \|_{ W^{1,\infty} (L^\infty(\Omega))}  \leq c \,. \label{esttz}
\end{align}
From~\eqref{zesti}, we deduce that the $\delta$-independent term on the left-hand side remain bounded in $n$ and $\varepsilon$, i.e., 
$\| z \|_{L^\infty (W^{1,2}(\Omega)} \leq c $. 
We note that the estimate~\eqref{estimate} remains true, i.e., is independent of $n$ and  $\varepsilon$. 

Testing Maxwell's equation with $\partial_t  \f A  $ gives 
\begin{multline*}
\int_0^T \int_D \sigma(\theta) | \t \f A |^2 \de \f x \de t +\left . \int_D  \frac{1}{ \mu(z)} | \curl \f A |^2 \de \f x \right |_0^T =\\
- \int_0^T \int_\Omega \frac{\mu'(z)\t z }{\mu^2(z)} |\curl \f A |^2 \de \f x \de t+\int_0^T \int_\Sigma  \f J_s \cdot \t \f A  \de \f x \de t \,.
\end{multline*}
Note that $ \t \mu=0$ on $ D / \Omega$. 
Young's and Gronwall's inequality implies
\begin{align}
\| \f A  \|_{ L^\infty( \f H_{\cur})} ^2  + \int_0^T \int_D \sigma(\theta) | \t \f A|^2 \de \f x \de t  \leq  C\left ( 1 + \| \f J_s \|^2_{L^2(0,T;L^2(\Sigma))}\right )    e^{ \int_0^T \| \t \log \mu (z) \|_{L^\infty(\Omega)} \de t } \leq c \,,\label{Areg}
\end{align}
where the right-hand side is bounded due to the boundedness of $\mu$ and~\eqref{esttz}.
From the left-hand side we infer due to the boundedness from below of $\sigma$ on $\Omega\cap \Sigma$ that
$ \| \t \f A \|_{L^2(0,T;L^2(\Omega\cap \Sigma))}\leq c $. 


Form the energy estimate~\eqref{energyregdis}, we get (since $\delta =0$)
\begin{multline}
\t \left ( \int_\Omega  e( \theta,z) \de \f x +\int_D \frac{1}{2\mu(z)} | \curl \f A |^2  \de \f x\right ) + \int_\Omega \frac{\mu'(z) \t z}{\mu^2(z)} |\curl \f A |^2 \de \f x + \int_{D/\Omega} \sigma_{\out} | \t \f A|^2 \de \f x   \\
=\int_D \f J_s \cdot \t \f A \de \f x +   \int_\Omega \frac{\varepsilon}{\theta^2} \de \f x  \,.\label{energydis}
\end{multline}
Multiplying by big enough constant $\gamma$ and adding the inequality~\eqref{entropiedis2}, we observe the boundedness of the right-hand side in a similar way as on page~\pageref{page:est} by absorbing the last term on the right-hand side of~\eqref{energydis} into the left hand side of~\eqref{entropiedis2}. 
This implies together  with~\eqref{esttz} and~\eqref{Areg} that
\begin{align}
\| e( \theta,z) \|_{L^\infty(L^1)} + \| \f A \|_{L^\infty(\f H_{\cur})}^2  \leq c \,.\label{ereg}
\end{align}
Additionally, the estimates of~\eqref{estimate} remain true (despite the last $\delta$-term). 


 We observe by comparison in the entropy production rate~\eqref{measureinterpretation}
that 
\begin{align*}
\| \t \psi_\theta (\theta , z) \|_{\mathcal{M}([0,T];(W^{-1,r})} \leq c \quad \text{for }r <3/2\,.
\end{align*}
Using again the chain rule, we may observe that (compare to Lemma~\ref{lem:thetae})
\begin{align*}
\nabla \psi_\theta (\theta ,z) = -\theta \psi_{\theta\theta} (\theta , z ) (-\nabla \log \theta ) + \psi_{\theta z} (\theta , z) \nabla z \,.
\end{align*}
With~\eqref{heatcapacity},~\eqref{estimate}, and~\eqref{zesti}, we conclude \begin{align*}
\| \nabla \psi _\theta (\theta ,z) \|_{L^2(L^{2})} \leq c \,.
\end{align*}

\subsection{Convergence of the approximate system\label{sec:conv2}}

By standard arguments, we infer the following weak$^*$ convergences
\begin{align}
%
%
%
%
z _ \varepsilon & \stackrel{*}{\rightharpoonup} z \quad \text{in } W^{1,\infty}( 0,T; L^\infty(\Omega)) \cap L^\infty (0,T ; W^{1,2}(\Omega))\,, \label{w:z}
\\
 \f A _\varepsilon & \stackrel{*}{\rightharpoonup}   \f A \quad \text{in } L^ \infty( 0,T; \f H_{\cur})
   \,, \label{w:A}
%
%
%
%
%
%
%
%
%
%
\\
\log \theta_\varepsilon &\stackrel{*}{\rightharpoonup}  \eta \quad \text{in } L^2 ( 0,T; H^1(\Omega))\cap L^\infty(0,T; \mathcal{M}(\ov\Omega
))  \,,\label{w:log}
\\
\psi_\theta (\theta_\varepsilon  ,z_\varepsilon)   &\stackrel{*}{\rightharpoonup}  \zeta \quad \text{in } L^2 ( 0,T; H^1(\Omega))\cap 
L^\infty(0,T;\mathcal{M}(\ov\Omega)) \,,
\\
%
%
%
%
%
%
%
%
%
\t \psi_\theta (\theta_\varepsilon , z_\varepsilon)  &\stackrel{*}{\rightharpoonup} \t \psi_\theta (\theta , z) \quad \text{in } \mathcal{M}([0,T];W^{-1,r}(\Omega)) \text{ for }r<3/2\,,
\\
\sqrt{\sigma_\varepsilon (\theta_\varepsilon) } \t \f A_\varepsilon & \rightharpoonup b  \quad \text{in } L^2(0,T; L^2(D ))\,,\label{sigtA}
\\
\t \f A_\varepsilon & \rightharpoonup \t \f  A   \quad \text{in } L^2(0,T; L^2(\Omega\cap \Sigma  ))\label{tAeps}
\,.
%
%
\end{align}
The Lions--Aubin lemma (see~\cite[Cor.~7.9]{roubicek} or~\cite[Thm.~3.22]{BV}) grants that 
\begin{align}
z _{\varepsilon} & \ra z \quad \text{in } L^p(0,T;L^p(\Omega)) \text{ for all }p\in [1,\infty)\,,\label{strongz}
\\
\f A_\varepsilon & \ra \f A \quad \text{in } L^2 (0,T ; \f L^2(\Omega\cap \Sigma )) \,,\label{strongA}
%
%
%
%
%
%
%
\\
 \psi_\theta (\theta_\varepsilon  ,z_\varepsilon) &\ra \zeta  \quad \text{in } L^2(0,T;L^2(\Omega) )\,. \label{psithetastrong}
\end{align}

To infer the strong convergence of $\log \theta_\varepsilon$, we observe that the mapping $\log  \theta \mapsto \partial _\theta \psi ( \theta,z)  $ has a continuous bounded inverse (see Lemma~\ref{lem:thetae}).   
Due to~\eqref{psithetastrong}, we can extract a subsequence that converges a.e. in $\Omega \times (0,T)$ and  a dominating function in $L^2(0,T; L^{2}(\Omega))$  to the sequence $\{ s_\varepsilon \} $ with $ s_\varepsilon :=\partial_\theta \psi(\theta_\varepsilon , z_\varepsilon)$ (see~\cite[Theorem~4.9]{brezisbook}). 
With Lemma~\ref{lem:thetae}, we may find a dominating function for the sequence~$\{ \log \theta_\varepsilon\}$. 
Expressing $\log \theta_\varepsilon $ via $ \log \theta_\varepsilon =  \hat s( s_\varepsilon, z_\varepsilon) $ (see Lemma~\ref{lem:thetae} for the definition of $\hat s$), the point-wise strong convergence of the logarithm of the temperatures follows from the continuity of the inverse function $\hat s$, the dominating function of $\{ s_\varepsilon\}$ together with Lebesgue's theorem on dominated convergence implies
%
%
%
%
%
%
%
%
\begin{align}
\log \theta_\varepsilon = \hat{s}(s_\varepsilon , z_\varepsilon )  & \ra \hat{s}(s,z) =   \log \theta  \text{ in } L^2 ( 0,T; L^2(\Omega)) \,\label{stronglog}
\end{align}
where $\theta$ is basically defined as the point-wise exponential of the limit. 

The continuity of $\psi_\theta$
allows to identify 
\begin{align}
 \psi_\theta (\theta_\varepsilon , z_\varepsilon )=\psi_\theta ( e^{\log \theta_\varepsilon}, z_\varepsilon)  \ra \psi_\theta ( \theta ,z)  \text{ in } L^2 ( 0,T; L^2(\Omega)) \,.\label{strongpsi}
\end{align}

 Together, we may infer different weak convergences
 \begin{align}
 \begin{split}
 \frac{\sqrt{\sigma(\theta_\varepsilon)}}{\sqrt{\theta_\varepsilon}} \t \f A_\varepsilon & \rightharpoonup  \frac{\sqrt{\sigma(\theta)}}{\sqrt{\theta}} \t \f A  \quad \text{in } L^2  (0,T; \f L^2(\Omega))\,, \\
 \sqrt{\kappa(\theta_\varepsilon , z_\varepsilon)} \nabla \log \theta_\varepsilon & \rightharpoonup   \sqrt{\kappa(\theta , z)} \nabla \log \theta  \quad \text{in } L^2  (0,T; \f L^2(\Omega))\,, \\ 
 \sqrt{\tau(\theta_\varepsilon)} \partial_t z_\varepsilon  & \rightharpoonup   \sqrt{\tau(\theta)} \partial_t z  \quad \text{in } L^2  (0,T; \f L^2(\Omega))\,,
 \end{split}\label{nonlinweak}
 \end{align}
where we were able to identify the nonlinear limits  due to the strong convergences~\eqref{strongz} and~\eqref{stronglog} and the weak convergences~\eqref{w:A},~\eqref{w:z}, and~\eqref{w:log} .  
From the fact that $\sigma_\varepsilon \ra 0$ in $D/(\Omega\cap\Sigma)$,~\eqref{sigtA},~\eqref{tAeps}, and~\eqref{stronglog}, we infer
\begin{align*}
\sigma_\varepsilon(\theta_\varepsilon) \t \f A_\varepsilon & \rightharpoonup  \sigma(\theta) \t \f A \quad \text{in } L^2  (0,T; \f L^2(D))\,.
\end{align*}
In the end, the positivity and continuity of the exponential function allows to conclude by Fatou's lemma:
\begin{align*}
0\leq \int_\Omega  \theta \de \f x  = \int_\Omega \exp( { \log \theta} )\de \f x \leq \liminf_{\varepsilon\ra 0}\int_\Omega \exp({\log \theta_\varepsilon}) \de \f x = \liminf_{\varepsilon\ra 0} \int_\Omega\theta_\varepsilon\de \f x \,.
\end{align*} 
Additionally, due to the positivity of $e$ and writing $ e ( \theta _\varepsilon , z_\varepsilon ) = e( \exp( \log \theta_\varepsilon) , z_\varepsilon)$, we find from the a.e.~convergence of $\log\theta_\varepsilon$ and Fatou's lemma (see Lemma~\ref{lem:thetae}) that
\begin{align}
\int_\Omega e(\theta(t) , z(t))\de \f x  \leq \liminf_{\varepsilon\ra 0} \int_\Omega e ( \theta_\varepsilon(t), z_\varepsilon(t))\de \f x  \label{weaklowen}
\end{align}
for a.\,e.~$t\in(0,T)$.

To be able to pass to the limit in the energy inequality~\eqref{energyequality}, we need to deduce strong convergence for $ \curl \f A_n$, since the sign of $\partial_z\mu(z)\t z $ is not known.
Note that we already established that the weak formulation of  Maxwell's equation holds, \textit{i.e.},~\eqref{weakA}.
Therefore, it is possible to subtract~\eqref{weakA} from~\eqref{magneticreg}. Since it is only possible to test this difference with appropriate test functions belonging to the discrete Galerkin space, we test with the difference $ \f A_n - P_n \f A$. This gives 
\begin{align*}
\int_0^T \int_D \left ( \sigma ( \theta_n )  \t \f A_{ n }  - \sigma( \theta ) \t \f A  )\cdot( \f A _n - P_n \f A \right ) + \left ( \frac{ 1}{\mu(z_n)} \curl \f A_n - \frac{1}{\mu(z) }\curl\f A \right )\cdot\left ( \curl \f A_n -\curl P_n \f A \right )\de \f x  = 0 \,.
\end{align*}
Some rearrangements provide the inequality:
\begin{multline*}
c \int_0^T  \int_\Omega | \curl \f A_n - \curl P_n \f A |^2 \de \f x \de t  \leq \int_0^T   \int_ \Omega \frac{1}{\mu(z_n)}| \curl \f A_n - \curl P_n \f A |^2 \de \f x \de t =\\  -\int_0^T  \int_{\Omega} \left (   \frac{1}{\mu(z_n) } \curl P_n \f A  - \frac{1}{\mu(z)} \curl P_n \f A + \frac{1}{\mu(z)} \curl P_n \f A - \frac{1}{\mu(z)} \curl \f A     \right ) \cdot \left ( \curl \f A_n - \curl P_n \f A \right ) \de \f x\de t  \\- \int_0^T \int_\Omega\left ( \sigma(\theta_\varepsilon)  \t \f A_\varepsilon  - \sigma( \theta )  \t \f A \right )\cdot \left ( \f A_n - P_n \f A \right ) \de \f x \de t  \,.
\end{multline*}
Due to~\eqref{strongz}, \eqref{stronglog},~\eqref{strongA}, and~\eqref{w:A}, the right-hand side of the previous equality converges to zero (note that $P_n \f A \ra \f A$ in $ L^2(0,T;\f H_{\cur})$) such that we infer
\begin{align}
\curl \f A_n & \ra \curl \f A \quad \text{in } L^2 (0,T ; \f L^2(D)) \,.\label{strongAgrad}
\end{align}
Similarly, without inserting the projection, we find the strong convergences as $ \varepsilon \ra 0$. 
From this, we may also infer some point-wise strong convergence and again, by Fatou's Lemma (see Lemma~\ref{lem:thetae}, we find
\begin{align}
   \left \| \frac{1}{\sqrt{2\mu(z(t))}}\curl \f A(t) \right \|_{L^2(D)}^2\leq \liminf_{\varepsilon\ra0} \left \| \frac{1}{\sqrt{2\mu(z_\varepsilon(t))}}\curl \f A_\varepsilon(t)\right  \|_{L^2(D)}^2\label{weaklowA}
\end{align}
for a.\,e.~$t\in(0,T)$.

In order to prove the energy inequality, we find 
with the last bound in~\eqref{estimate} that
\begin{align*}
\int_{\Omega}\frac{\varepsilon}{\theta_\varepsilon^2  } \de \f x = \int_\Omega \varepsilon ^{1/3} \left (\frac{\varepsilon^{1/3}}{\theta_\varepsilon}\right )^2 \de \f x  \leq  \varepsilon ^{1/3} | \Omega|^{1/3} \left \| \frac{\varepsilon^{1/3}}{\theta_\varepsilon}\right \|_{L^3}^2 \ra 0 \text{ as } \varepsilon\ra 0\,,
\end{align*}
which proves that the second term on right-hand side of~\eqref{energydis} vanishes as $\varepsilon\ra 0$.
 Integrating equation~\eqref{energydis} in time, we find
\begin{multline}
\left . \left (\int_\Omega e( \theta_\varepsilon,z_\varepsilon) \de \f x +\int_D  \frac{1}{2\mu(z)} | \curl \f A _\varepsilon|^2  \de \f x\right ) \right |_0^t+ \int_0^t   \int_\Omega \frac{\mu'(z_\varepsilon) \t z_\varepsilon}{\mu^2(z_\varepsilon)} |\curl \f A _ \varepsilon|^2 \de \f x \de t \\   + \int_0^t  \int_{D/\Omega} \sigma_\varepsilon| \t \f A_\varepsilon|^2\de \f x \de t   = \int_0^t \zeta \int_\Omega \frac{\varepsilon}{\theta_\varepsilon^2}\de \f x \de t+ \int_0^T \zeta \int_\Sigma  \f J_s \cdot \t \f A \de \f x \de t  \,
\end{multline}
for a.\,e.~$t\in(0,T)$.
The weak-lower semi-continuities~\eqref{weaklowen} and~\eqref{weaklowA}, the strong convergences~\eqref{strongz} and~\eqref{strongAgrad} as well as the weak convergence~\eqref{tAeps} allow to go to the limit in this equation and implies the energy inequality~\eqref{energyequality} a.e.~in $(0,T)$.


With the bound~\eqref{estimate}, we find  that 
\begin{align}
\begin{split}
\int_\Omega \varepsilon \theta^{p-1}_\varepsilon \nabla \theta_\varepsilon \cdot \nabla  \vartheta \de \f x \leq{}&  c \varepsilon^{1/2}  \| \theta_{\varepsilon}^{p/2} \|_{L^2(L^2)}  \varepsilon^{1/2} \| \nabla \theta_\varepsilon^{p/2}\|_{L^2 (L^2)} \| \nabla \vartheta\|_{L^\infty(L^\infty)}\\ \leq{}& c \varepsilon^{1/2}  \| \theta_{\varepsilon}^{p/2} \|_{L^2(W^{1,2})}^{1-2/(3p-1)} \| \theta _{\varepsilon} ^{p/2}\|_{L^\infty(L^{2/p})}^{2/(3p-1)} \varepsilon^{1/2} \| \nabla \theta_\varepsilon^{p/2}\|_{L^2 (L^2)} \| \nabla \vartheta\|_{L^\infty(L^\infty)}
\\ \leq{}& c \varepsilon^{1/(3p-1)} \left ( \varepsilon^{1-1/(3p-1)}  \| \theta_{\varepsilon}^{p/2} \|_{L^2(W^{1,2})}^{2-2/(3p-1)}\right ) \| \theta _{\varepsilon} \|_{L^\infty(L^{1})}^{p/(3p-1)} \| \nabla \vartheta\|_{L^\infty(L^\infty)}\,,
\end{split}\label{vanishreg}
\end{align}
where we used H\"older's inequality in the first line, Gagliardo-Nirenberg's inequality in the second estimate and just a rearrangement in the last one.
Since the term in the brackets on the right-hand side of the previous estimate is bounded due to~\eqref{estimate}, the right-hand side vanishes as $\varepsilon\ra 0$. 
Now we can conclude for the entropy inequality:
Testing equations~\eqref{tempreg} with $ -\vartheta/ \theta $ with a positive function $\vartheta\geq 0$, $\vartheta \in \C^1_c(\ov\Omega) \otimes \C^1 ([0,T])$, we find the inequality
%
\begin{align}
\begin{split}
\left . \int_\Omega \psi _{\theta }( \theta_\varepsilon , z_\varepsilon) \vartheta\de \f x \right |_0^t + &\int_0^t \int_\Omega \vartheta
 \left (  \kappa( \theta_\varepsilon , z_\varepsilon) | \nabla \log \theta _\varepsilon|^2 + \sigma( \theta_\varepsilon ) 
 \frac{| \t \f A _\varepsilon   |^2}{\theta _\varepsilon } + \tau(\theta ) \frac{(\t z_\varepsilon)^2}{\theta_\varepsilon} \right ) \de \f x \de s \\
& +\varepsilon \int _0^t \int_\Omega  \left ( | \nabla \theta ^{p/2} |^2  +  \frac{1}{\theta_\varepsilon^3}  \right )  
\vartheta  -  \theta_\varepsilon ^{p-1} \nabla \theta_\varepsilon \cdot \nabla \vartheta  \de \f x \de s \\
& - \int_0^t \int_\Omega \kappa(\theta _\varepsilon , z_\varepsilon ) \nabla \log \theta_\varepsilon \cdot \nabla \vartheta   \de \f x \de s  \leq  \int_0^t \int_\Omega \psi_\theta ( \theta_\varepsilon,z_\varepsilon ) \t \vartheta   \de \f x \de s 
\end{split}
\label{entropyregularized}
\end{align}
for a.\,e. $t\in (0,T)$ and $\vartheta\geq 0$, $\vartheta \in \C^1_c(\ov\Omega) \otimes \C^1 ([0,T])$.
%
%
Due to~\eqref{vanishreg}, the last term in the second line of~\eqref{entropyregularized} vanishes as $\varepsilon\ra 0$ the other two terms in the second line of~\eqref{entropyregularized} can be estimated from below by zero.  The convergences~\eqref{nonlinweak} and the weak-lower semi-continuity of convex functions (\textit{cf}.~\cite[Thm.~10.20]{singular} or~\cite{ioffe}) allows to go to the limit in the first line of~\eqref{entropyregularized} with exception of the first term in that line.
For  the term on the right-hand side of~\eqref{entropyregularized}, we observe strong convergence from~\eqref{stronglog}. 
Going to the limit in the entropy term  $\psi_\theta$ in the first line of~\eqref{entropyregularized} in done in the same way as in the delta limit on page~\pageref{entropiedis}.

Finally, we conclude 
 that the limit fulfills the entropy inequality~\eqref{entropy} 
and additionally, it fulfills  the measure-valued formulation~\eqref{measureinterpretation} with $\varepsilon=0$.

%

The weak convergences~\eqref{w:A} and the strong convergences~\eqref{strongtemp} and~\eqref{strongz} as well as the assumptions on $\sigma $ and $\mu$ allows to go to the limit in~\eqref{magneticreg} and attain the weak formulation~\eqref{weakA}. 
Note that during the last step, as $\varepsilon $ tends to zero, we simultaneously go to the limit in the regularization (denoted by $^\varepsilon$ in~\eqref{eq:reg}). Since these regularization converge in $\C$, the strong convergences~\eqref{strongz} and~\eqref{strongtemp} allows to identify the limits appropriately.

\subsection{Additional estimates in the case of more regularity of the free energy\label{sec:weakss}}
In this section, we prove the assertion of Remark~\ref{rem:add}. 
In an additional estimate, we want to test the energy equation with the a function $  - \theta ^{\alpha-1} $ for a $ \alpha\in (1/2,1)$.
We define $ C_\alpha( \theta,z) $ by 
$ C_\alpha( \theta,z) =-  \int_0^\theta r^\alpha \psi_{\theta ,\theta} (r , z ) \de r $
such that 
\begin{align*}
-\t C_{\alpha} ( \theta , z) ={}& \theta ^ \alpha \psi_{ \theta ,\theta} ( \theta , z)\partial_t \theta + \int_0^\theta r^\alpha \psi_{\theta ,\theta, z } (r , z ) \de r  \partial _t z\\ ={}& \theta^\alpha \partial_t \psi_\theta( \theta ,z) - \theta^\alpha \psi_{\theta ,z} ( \theta , z)  \t z + \int_0^\theta r^\alpha \psi_{\theta ,\theta, z } (r , z ) \de r  \partial _t z \,
\end{align*} 
and $C_\alpha (\theta ,z) \leq c(\theta^\alpha+1) \leq c ( \theta +1) $ is bounded by~\eqref{ereg}.
Note that $C_\alpha$ is positive due to the concavity of $\psi $ in the argument $\theta$. 
We observe for the energy equation tested with  $  - \theta ^{\alpha-1} $ that
\begin{multline}
\frac{4}{\alpha^2 } ( 1- \alpha) \int_0 ^T \int_\Omega | \nabla \theta ^{\alpha/2}|^2 \de \f x \de t+\frac{4\varepsilon }{(\alpha+2)^2 } ( 1- \alpha) \int_0 ^T \int_\Omega | \nabla \theta ^{(\alpha+2)/2}|^2 \de \f x \de t \\ + \int_0^T \int_\Omega \theta^{\alpha-1} \left ( \sigma(\theta) | \t \f A |^2 +\frac{ \tau(\theta)}{2}  | \partial _t z|^2+ \frac{\delta^{3/2}\tau(\theta)}{2}| \Delta z|^2 +\frac{\varepsilon}{\theta_\varepsilon^{\alpha-3}} \right ) \de \f x \de t  
\\ \leq \left .\int_\Omega C_{\alpha}(\theta , z) \de \f x \right |_0^T + \int_0^T \int_\Omega \left (\theta^\alpha \psi_{\theta ,z} ( \theta , z)  - \int_0^\theta r^\alpha \psi_{\theta ,\theta, z } (r , z ) \de r\right ) \t z \de \f x \de t 
\leq c  +c \int_0^T \int_\Omega {\theta^{\alpha/2}} ^p \de \f x \de t  
\,.\label{notunderweak}
\end{multline}
To observe that the right-hand side is bounded by $ c  +c \int_0^T \int_\Omega {\theta^{\alpha/2}} ^p \de \f x \de t  $,~\eqref{esttz} is again essential.
Additionally, we used 
\begin{align*}
\int_0^T\int_\Omega \theta^\alpha\psi_{\theta z} (\theta ,z) \de \f x \de t \leq{}& C \int_0^T\int_\Omega ( \sqrt{\theta } + \theta ) \psi_{\theta z} (\theta ,z) \de \f x \de t  \leq C
 \intertext{and}
\int_0^T\int_\Omega  \int_0^\theta r^\alpha \psi_{\theta ,\theta, z } (r , z ) \de r \de \f x \de t  \leq{}& C\int_0^T\int_\Omega\int_0^\theta r^{\alpha-{s}} \de r \de \f x \de t  \leq C\int_0^T\int_\Omega (\theta ^{\alpha -s} +1) \de \f x \de t\,.
\end{align*}
 Such that $p$ in~\eqref{notunderweak} is given by $p <2(\alpha + 2/3)/\alpha = 2 + 4/(3\alpha)$ (cf.~\eqref{psithetas}). 
The Gagliardo--Nirenberg inequality implies
\begin{align*}
\| \theta ^{\alpha/2}\|_{\f L^p(\Omega\times (0,T)) }^p \leq c(  \|  \nabla \theta^{\alpha/2} \|_{ L^2(\Omega\times (0,T))} ^ \beta + \| \theta^{\alpha/2}\|_{L^{2/\alpha}(\Omega\times (0,T))}^\beta  ) \| \theta ^{\alpha/2} \|_{L^{2/\alpha}(\Omega\times (0,T))}^{p-\beta}\, 
\end{align*}
for $  1/p = \beta / 6p   +(1-\beta / p)   \alpha /2 $ such that $ \beta = 3(p \alpha -2)/ ( 3 \alpha - 1) $ and for $p < 2 + 4/(3\alpha)$ we find that $ \beta < 2$.  
 Young's inequality implies 
\begin{align*}
C \| \theta ^{\alpha /2}\|_{L^p(\Omega\times (0,T))}^p \leq\frac{2(1-\alpha)}{\alpha^2} \| \nabla \theta^{\alpha/2} \|_{L^2(\Omega\times (0,T))}^2 +C  \| \theta^{\alpha /2} \|_{L^{2/\alpha}(\Omega\times (0,T))} ^ q \,, 
\end{align*}
where $q$ is chosen accordingly. We infer  that  the gradient term can be absorbed in the left-hand side of the inequality~\eqref{notunderweak}.
Note that $ \| \theta ^{ \alpha /2} \|_{ L^{2/\alpha}}^{2/\alpha} = \| \theta \|_{ L^1} $, which is already bounded due to the energy estimate and $ \| \theta \| _{ L^1} \leq c \| e \|_{ L^1}$ (compare with~Lemma~\ref{lem:thetae}). 
Together, we find
\begin{align*}
\| \theta ^{\alpha/2} \|_{L^2( W^{1,2})} \leq c \,.
\end{align*}
Note that these estimates are now independent of $\varepsilon$.

First, we observe that $ \theta^{\alpha/2} \in L^2(0,T;L^6(\Omega)) \cap L^\infty(0,T;L^{2/\alpha}(\Omega))$, which implies by an interpolation inequality that $\theta^{\alpha/2} \in L^{2(2+3\alpha)/3\alpha} (\Omega \times (0,T)) $ and thus
\begin{align*}
\theta \in L^p(\Omega \times (0,T)) \quad \text{for all } p\in [1,5/3)\,.
\end{align*} 


To get an estimate for $\nabla\theta$ itself, we use a calculation similar to~\cite{boccardo,tomassetti,malek}.
 We recall a version of the Gagliardo--Nirenberg inequality
\begin{align}
\| \theta \| _{ L^{3q/2} }^{3q/2} \leq c (\| \nabla \theta \| _{ L^q} ^ q + \| \theta \|_{ L^1}^q) \| \theta \|_{L^1}^{q/2}\,.\label{gag}
\end{align}
For $q$ such that $1\leq q <4/3$, there exists an $\alpha \in (1/2,1)$ such that $ (2-\alpha) /(2-q) = 3/2$. 
We may conclude
\begin{align*}
\| \nabla \theta \| _{ L^q(L^q)} ^q \leq{}& \left (\int_0^T \int_\Omega \frac{| \nabla \theta|^2}{\theta^{2-\alpha}} \de \f x \de t \right )^{q/2} \left ( \int_0^T \int_\Omega  \theta^{q(2-\alpha)/(2-q)} \de \f x \de t \right ) ^{(2-q)/2} \\
 ={}& \left (\frac{4}{\alpha^2}\right )^{q/2}  \| \nabla \theta^{\alpha/2}\|_{L^2(L^2)}^{q/2}  \| \theta\|_{L^{3q/2}}^{3q(2-q)/4}
\\ 
 \leq{}&c  \left  (\| \nabla \theta^{\alpha/2}\|_{L^2(L^2)} \| \theta \|_{L^\infty(L^1)}^{(2-q)/2}\right )^{q/2} (\| \nabla \theta \| _{ L^q(L^q)} ^ {q} + \| \theta \|_{ L^\infty(L^1)}^q)^{(2-q)/2}
 \\
 \leq {} & \frac{1}{2} \left (\| \nabla \theta \| _{ L^q(L^q)} ^ {q} + \| \theta \|_{ L^\infty(L^1)}^q\right )  + C  \left  (\| \nabla \theta^{\alpha/2}\|_{L^2(L^2)} \| \theta \|_{L^\infty(L^1)}^{(2-q)/2}\right )
 \,.
\end{align*}
The first of the above inequalities is due to H\"older's inequality, the equality is just the identification of the norm an the equality $ (2-\alpha) q/(2-q) = 3q/2$. The second inequality follows from Gagliardo--Nirenberg's inequality~\eqref{gag} and the last one due to Young's inequality $c ( a \cdot b) \leq( 1/2) a ^ {2/(2-q)} + \frac{1}{2}(c b )^{ 2/q} $.
We observe that the first term on the right-hand side of the last inequality can be absorbed in the left-hand side such that
\begin{align*}
\| \nabla \theta \| _{ L^q(L^q)}^q  \leq  C  \left  (\| \nabla \theta^{\alpha/2}\|_{L^2(L^2)} \| \theta \|_{L^\infty(L^1)}^{(2-q)/2}+ \| \theta \|_{ L^\infty(L^1)}^q\right ) \,.
\end{align*} 


%
Concerning the limit in~$n$, we observe
by comparison in~\eqref{energyreg}  with $ | \t \f A|^2$ being bounded in $  L^1$  that 
\begin{align}
\| \t e \|_{ \mathcal{M} ([0,T] W ^{-1,q} )} \leq c \qquad \text{with } q \in ( 1, 4/3)\,.\label{eesttime}
\end{align}
Since we are not able to establish strong convergence for $\{ \t \f A_n\}$ as $n \ra 0$ (unless $\partial_z \mu \equiv 0$),~\eqref{energyreg} cease to hold in the limit. It only holds that there exists a measure $\xi _\varepsilon\in \mathcal{M}^+(\Omega \times (0,T) $ such that
\begin{align}
- \int_0^T\int_\Omega e(\theta_\varepsilon,z_\varepsilon) \t \zeta  \de \f x \de t + \int_0^T \int _\Omega \kappa_\varepsilon (\theta_\varepsilon ,z_\varepsilon) \nabla \theta_\varepsilon \cdot \nabla \zeta \de \f x \de t  =\int_0^T \int_\Omega  \zeta \de \xi_\varepsilon + \frac{\varepsilon}{\theta_\varepsilon ^2}\zeta \de \f x\de t \,\label{energynotdis}
\end{align}
for all $ \zeta \in \C_c^\infty(\Omega \times (0,T))$. 
We additionally find 
\begin{align*}
\int_0^T \int_\Omega  \zeta \de \xi _\varepsilon \geq \int_0^T \int_\Omega \sigma(\theta_\varepsilon) | \t \f A_\varepsilon |^2\zeta   \de \f x \de t 
\end{align*}
for all  $ \zeta \in \C_c^\infty(\Omega \times (0,T))$ with $ \zeta \geq 0$. From~\eqref{energynotdis}, we observe that~\eqref{eesttime} also holds independently of $\varepsilon$.

The chain rule implies
\begin{align*}
\n e ( \theta , z) = \n ( \psi( \theta , z) - \theta  \psi_\theta ( \theta ,z) ) = \psi _z ( \theta , z) \n z - \theta  \psi_{\theta ,\theta} ( \theta ,z) \n \theta -\theta \psi_{\theta,z} ( \theta ,z) \n z   \,
\end{align*}
such that we can estimate due to the properties of the free energy~$\psi$ (see~\eqref{heatcapacity})
\begin{align*}
\|\n e (  \theta ,z) \| _{ L^q(L^q)} \leq c (\|\nabla z \|_{ L^2( L^2)} + \| \nabla \theta\|_{ L^q ( L^q)} ) \leq c   \qquad \text{for }q \in [1,4/3)\,.
\end{align*}

We find the additional weak convergences  and due to the generalized Lions--Aubin lemma (see~\cite[Cor.~7.9]{roubicek} or~\cite[Thm.~3.22]{BV}) that
\begin{align*}
e_\varepsilon & \rightharpoonup e \quad \text{in } L^q(0,T; W^{1,q} )\cap L^p(0,T;L^p)  \text { for } q\in[1,4/3) \text{ and } p\in [1,5/3)\,,\\
\t e_\varepsilon &\stackrel{*}{\rightharpoonup} \t e \quad\text{in } \mathcal{M}([0,T];W^{-1,q})\,,\\
e_\varepsilon & \ra e \quad \text{in } L^p ( 0,T ; L^{p}(\Omega))  \text{ for } p \in [1,5/3)  \,,.
\end{align*}
The same arguments as the ones after formula~\eqref{strongedis} show  the strong convergence of $\{\theta_\varepsilon\}$, i.e.,
 such that 
\begin{align}
\theta_{\varepsilon} =\hat e  (e_ \varepsilon , z_\varepsilon)&\ra \hat e (e,z) = \theta \quad \text{in }L^p(0,T; L^{p}(\Omega)) \quad \text{for }p\in [1,5/3)\,.\label{strongtemp}
\end{align}
\begin{remark}
Nearly all terms in the energy inequality~\eqref{energydis} converge strongly despite the last term on the left-hand side of~\eqref{energydis} incorporating the time derivative of $\f A$. This term is an additional dissipative term, stemming from the fact that (for simplicity), we do not consider the energy balance outside of the work piece. Considering only the system in $\Omega$, we would even be able to show the energy equality under the additional assumption of Remark~\eqref{rem:add}. 
\end{remark}


\section{Weak-strong uniqueness\label{sec:weakstrong}}
This section is devoted to the proof of Theorem~\ref{thm:weakstrong}.
\subsection{Relative energy and associated estimates\label{sec:relen}}
To abbreviate, we introduce the vectors $\f u$ and $\tu$ by $ \f u = ( \theta , z ,\curl \f A^T)^T$ and $\tu = ( \tet, \tz, \curl \tA^T)^T$. 
The relative energy is defined by 
\begin{align}
\begin{split}
 \mathcal{E}(\f u |\tu)  :={}& \int_\Omega  \psi( \theta,z  ) - \theta \psi_\theta (\theta,z ) \de \f x + \int_D\frac{1}{2\mu(z)} | \curl \f A|^2 \de \f x 
\\& 
 - \int_\Omega  \psi(\tet,\tz ) - \tet \psi_\theta (\tet,\tz )\de \f x +\int_D \frac{1}{2\mu(\tz)} | \curl \tA|^2 \de \f x \\
& - \int_\Omega \psi_z ( \tet,\tz) (z-\tz) \de \f x  + \int _ \Omega \tet ( \psi_\theta ( \theta ,z  ) - \psi_\theta (\tet, \tz ) )\de \f x 
\\ & -\int_D\frac{1}{{\mu(\tz)}} \curl \tA  \cdot \left (  \curl \f A -  \curl \tA    \right ) - \frac{\mu'(\tz) }{2\mu(\tz)^2} | \curl \tA |^2 (z-\tz)  \de \f x 
\,.
\end{split}\label{relativeen}
\end{align}
and the relative dissipation by
\begin{align*}
\mathcal{W}
( \f u ,\tu)
:= {}&
\int_\Omega 
\frac{\tet}2 \kappa( \theta ,z) |\nabla \log \theta - \nabla \log \tet|^2+ \frac{\sigma(\theta ) }{2} \left |   \sqrt{\frac{\tet}{\theta} }  \t \f A  -\sqrt{\frac{\theta}{\tet} }  \t \tA   \right |^2
\de \f x
\\
&
+ \int_\Omega  \frac{ \tau(\theta)}{2}\left (\sqrt{\frac{\tet}{\theta} }  \t z -\sqrt{\frac{\theta}{\tet} }  \t \tz\right )^2 
\de \f x + \int_{D/\Omega} \sigma _{\out} | \t \f A - \t \tA|^2 \de \f x \,. 
\end{align*}

\begin{proposition}\label{prop:weakstrong}
Let $( \theta ,\f A ,z )$ be a weak entropy solution (see Definition~\ref{def:weak}) and let $ (\tet,\tA, \tz)$ be a strong solution.
Then the relative energy inequality 
\begin{multline}
\mathcal{E}(\theta , \f A ,z | \tet,\tA, \tz) (t) + \int_0^t  \mathcal{W}(\theta , \f A ,z | \tet,\tA, \tz)  \exp\left ({\int_s^t \mathcal{K}(\tau) \de \tau}\right ) \de s 
\\
\leq \mathcal{E}(\theta , \f A ,z | \tet,\tA, \tz) (0) \exp \left ( {\int_0^t \mathcal{K}(s)\de s } \right )
\,,\label{relenergyin}
\end{multline}
holds for a.e.~$t\in(0,T)$, where $ \mathcal{K}$ is given by
\begin{align*}
\mathcal{K}(t) ={}& c \left ( \| \t \log \tet(t)\|_{L^\infty}+ \| \t \tz(t) \|_{L^\infty} + \| \curl \t \tA(t) \|_{L^\infty} + \| \nabla \log \tet(t)\|_{L^\infty }^2 \right ) \\ & +c \left (  \| \di ( \kappa(\tet(t) ,\tz(t)) \nabla \log \tet(t) ) \|_{L^\infty} + \| \t \log \mu(z(t)) \|_{L^\infty}  + \| \curl\tA(t) \|_{L^\infty} ^2  \right )\,
\end{align*}
and $c$ depends on different norms of the regular solutions, i.e., $ \| \curl \tA \|_{L^\infty(L^\infty)}$, $ \| \tet \|_{L^\infty(L^\infty)} $, and~$ \| \tet^{-1}\|_{L^\infty(L^\infty)}$, as well as the regularity of the associated material functions (cf.~Assumptions~\ref{Asume:1}). Note that $z$, $\t z$, as well as $\tz$, $\t \tz$ are bounded in the $L^\infty(0,T;L^\infty(\Omega))$-norm due to the regularity of the free energy~$\psi$ (see~\eqref{weakz}).  
\end{proposition}

The proof of this theorem is executed in the remaining part of this section.
First, 
we collect several preliminary results.
\begin{proposition}\label{prop:pos}
Let the assumptions of Theorem~\ref{thm:weakstrong} be fulfilled. Then the relative energy can be estimated from below by
\begin{align*}
\mathcal{E}(\theta , \f A ,z | \tet,\tA, \tz)  \geq c \int_\Omega (\theta -\tet -\tet( \log \theta - \log \tet) ) + | z-\tz|^2 \de \f x + \int_D | \curl \f A - \curl \tA|^2  \de \f x  \geq 0 \,
\end{align*}
for all $ (\theta , z,\f A) $ fulfilling~\eqref{regweak} and $ ( \tet,\tz,\tA)$ fulfilling~\eqref{regregular}. 
\end{proposition}
\begin{remark}
From $  \curl \f A = \curl \tA$ follows $\f A = \tA$ since both are in the space $\f H_{\cur}$. 
\end{remark}
\begin{proof}
We recall the assumptions~\eqref{heatcapacity} and~\eqref{condmu}.
Sorting the convex and concave parts yields
\begin{align}
\begin{split}
\mathcal{E}&(\f u  | \tu)  \\
={}& \int_\Omega \psi( \tet , z) -\psi(\tet, \tz )  - \psi_z ( \tet, \tz ) (z-\tz)  \de \f x
- \int_\Omega  \psi( \tet , z) - \psi( \theta , z) - \psi_\theta( \theta ,z ) ( \tet - \theta ) \de \f x
\\
&+ \int_{D}  \frac{1}{2\mu(z)} | \curl \f A|^2 - \frac{1}{2\mu(\tz)} | \curl \tA|^2 -  \frac{1}{{\mu(\tz)}} \curl \tA  \cdot \left (  \curl \f A -  \curl \tA    \right ) + \frac{\mu'(\tz) }{2\mu(\tz)^2} | \curl \tA |^2 (z-\tz) \de \f x \,.
\end{split}\label{relenpos}
\end{align}
The first line of the right-hand side of the forgoing equality is non-negative due to the convexity of $ z \mapsto \psi( \tet, z) $ for every $ \tet \in (0,\infty)$. 
The first line of~\eqref{relenpos} can be transformed using Taylor's formula (cf.~Assumption~\ref{Asume:1}).
\begin{align*}
\psi( \tet , z) -\psi(\tet, \tz )  - \psi_z ( \tet, \tz ) (z-\tz) = \int_0^1 \psi_{zz} ( \tet, \tz + s(z-\tz)) \de s (z-\tz)^2 \geq c (z-\tz)^2 \,.
\end{align*}
For the second line of~\eqref{relenpos}, we observe with a rearrangement
\begin{align*}
- (\psi( \tet , z) - \psi( \theta , z) - \psi_{\theta}( \theta ,z ) ( \tet - \theta )  )
={}& \psi(\theta , z) - \theta \psi_{\theta }( \theta ,z) - \left ( \psi(\tet,z) - \tet \psi_{\theta} (\tet , z) \right ) \\
&+  \tet ( \psi_\theta( \theta, z) - \psi_\theta( \tet ,z) ) 
\end{align*}
that the first two terms on the right-hand side correspond to the energy. Since the derivative of the energy is given by $- \theta \psi_{\theta\theta}$, the right-hand side can be expressed via
\begin{align*}
- (\psi( \tet , z) - \psi( \theta , z) - \psi_{\theta}( \theta ,z ) ( \tet - \theta )  )
 ={}& -\int_{\tet} ^\theta r \psi_{\theta\theta} (r ,z) \de r + \tet \int_{\tet}^\theta \psi_{\theta\theta}(r,z) \de r \\
 ={}&  \int_{\tet}^\theta \left (1- \frac{\tet  }{r } \right )\left (- r \psi_{\theta\theta}(r,z) \right ) \de r \,.
\end{align*} 
With the bound~\eqref{heatcapacity}, we observe
\begin{align*}
- (\psi( \tet , z) - \psi( \theta , z) - \psi_{\theta}( \theta ,z ) ( \tet - \theta )  )
 \geq  c  ( \theta - \tet -\tet( \log \theta -\log \tet ) ) \,.
\end{align*}
Note that the right-hand side of the previous estimate is positive, indeed due to the convexity of the exponential function, we may infer \begin{align*}
 \theta - \tet -\tet( \log \theta -\log \tet )  = \exp ({\log \theta }) - \exp ({\log \tet}) - \exp'(\log\tet) ( \theta - \tet) \geq 0 \,,
\end{align*}
where the quantity $ f (x) - f( y) - f'( y) (x-y)>0$  is positive for every convex function $f\in\C^1(\R)$ and all $ x$, $y\in \R$ with $x\neq y$. 
The last line of~\eqref{relenpos} is positive  on $D/\Omega$ since $\mu$ is constant such that $\partial_z \mu\equiv 0$ and the last line in~\eqref{relenpos} simplifies to
\begin{multline}
\int_{D/\Omega}  \frac{1}{2\mu(z)} | \curl \f A|^2 - \frac{1}{2\mu(\tz)} | \curl \tA|^2 -  \frac{1}{{\mu(\tz)}} \curl \tA  \cdot \left (  \curl \f A -  \curl \tA    \right ) + \frac{\mu'(\tz) }{2\mu(\tz)^2} | \curl \tA |^2 (z-\tz) \de \f x  
\\= \int_{D/\Omega} \frac{1}{2\mu_{\out}} | \curl \f A - \curl \tA|^2 \de \f x \,.\label{DoOm}
\end{multline}
The last line in~\eqref{relenpos} is also positive on~$\Omega$  since the mapping $ (z ,\f B) \mapsto | \f B|^2 / (2\mu(z)) $ is convex. 
Indeed, due to the regularity of $\mu$ (see~\eqref{condmu}), we may compute its second derivative.
Defining $ g( z ,\f B) = | \f B|^2 / (2\mu(z)) $, we find
\begin{align*}
\nabla ^2_{(z,\f B)} g( z ,\f B) = \begin{pmatrix}
\left (- \frac{\mu''(z) }{2\mu^2(z) } + \frac{| \mu'(z)|^2}{\mu^3(z)}\right ) | \f B|^2 & - \frac{\mu'(z)}{\mu^2(z)} \f B ^T  \\   - \frac{\mu'(z)}{\mu(z)} \f B  & \frac{1}{\mu(z)} I 
\end{pmatrix}\,.
\end{align*}
Concerning the positive definiteness of this matrix, we observe 
\begin{align}
\begin{split}
\begin{pmatrix}
y \\ \f D 
\end{pmatrix}
\cdot \left (
 \nabla ^2_{(z,\f B)} g( z ,\f B) \right ) \begin{pmatrix}
y \\ \f D 
\end{pmatrix}
&=  \left (- \frac{\mu''(z) }{2\mu^2(z) } + \frac{| \mu'(z)|^2}{\mu^3(z)}\right ) | \f B|^2 y^2 - 2 \frac{\mu'(z)}{\mu^2(z)} \f B \cdot \f D y + \frac{1}{\mu(z)}  | \f D|^2 \\
&\geq   \left (- \frac{\mu''(z) }{2\mu^2(z) } + \frac{| \mu'(z)|^2}{\mu^3(z)}\right ) | \f B|^2 y^2   -  \frac{3|\mu'(z)|^2}{2\mu^3(z)}| \f B|^2 y^2 -\frac{2}{3\mu(z)}  | \f D|^2+ \frac{1}{\mu(z)}  | \f D|^2 \\
&=   \left (- \frac{\mu''(z) }{2\mu^2(z) } - \frac{| \mu'(z)|^2}{2\mu^3(z)}\right ) | \f B|^2 y^2  + \frac{1}{3\mu(z)}  | \f D|^2 \\
&=  - \frac{(\mu^2(z))'' }{4\mu^3(z) }  | \f B|^2 y^2  + \frac{1}{3\mu(z)}  | \f D|^2 \geq \frac{1}{3\ov \mu } | \f D|^2\,,
\end{split}\label{second}
 \end{align}
where we used Young's inequality and the assumption $ - ( \mu^2(z))'' \geq 0 $ for all $z \in \R$.  
The  last line in~\eqref{relenpos} can be interpreted with Taylor's formula as
\begin{multline*}
\int_\Omega g( z ,\curl \f A ) - g( \tz ,\curl \tA ) - \nabla _{ (z,\f B) } g(\tz ,\curl \tA) \cdot  \begin{pmatrix}z-\tz\\
\curl \f A - \curl \tA  
\end{pmatrix} \de \f x = \\ \int_\Omega
\begin{pmatrix}
z-\tz\\ \curl \f A - \curl \tA 
\end{pmatrix}
\cdot  \int_0^1 \left (
 \nabla ^2_{(z,\f B)} g( \tz+ s(z-\tz) ,\curl \tA + s( \curl \f A - \curl \tA)) \right ) \de s  \begin{pmatrix}
z-\tz\\ \curl \f A - \curl \tA 
\end{pmatrix}  \,.
\end{multline*}
The estimate~\eqref{second} implies the last bound, i.e.,
 \begin{align*}
  \mathcal{E}(\theta , \f A ,z | \tet,\tA, \tz) \geq \frac{1}{3\ov\mu }\int_\Omega  | \curl \f A - \curl \tA|^2 \de \f x \,.
 \end{align*}

\end{proof}
\begin{lemma}\label{lem:est}
Let $f\in \C^1([0,\infty);\R)$ be a function that is bounded and the derivative is also bounded and vanishes sufficiently fast at infinity, i.e.,
\begin{align*}
f ( x) \leq C \,, \quad |f '( x)| (x^3 +  x + 1) \leq c \quad \text{ for all } x \in [0 ,\infty)
\,.   
\end{align*}
Then the estimate
\begin{align}
|f(x ) - f( y)|^2(1+( \log (x) - \log(y) )  \leq \frac{c}{y} (x-y-y(\log x -\log y))\,. \label{logest}
\end{align}
holds for all $x,y\in [0,\infty) $ with $y >0$. 
\end{lemma}
\begin{remark}
The material functions  $\tau$, $\sigma$, and $ \kappa(\cdot , z)$ comply to the assumptions of Lemma~\ref{lem:est} (compare to Assumption~\ref{Assume:2}).
\end{remark}
\begin{proof}
Using the properties of the exponential function, we may rewrite the right-hand side of~\eqref{logest} 
\begin{align*}
x-y-y(\log x -\log y) = \exp({\log x} )- \exp({\log y}) - \exp( { \log y }) ( \log x -\log y) \,.
\end{align*}
The integral representation of the Taylor expansion up to first order implies
\begin{align}
\begin{split}
x-y-y(\log x -\log y)  ={}&\frac{1}{2} \int_0^1 (1-s) \exp({\log y + s( \log x-\log y)}) \de s ( \log x - \log y )^2  \\
\geq {} & \frac{1}{4} \int_0^{\frac{1}{2}}  \exp({\log y + s( \log x-\log y)}) \de s ( \log x - \log y )^2 \\
={}& \frac{1}{4}\left[ \exp({\log y + s( \log x-\log y)}) \right]_0^{\frac{1}{2}}  ( \log x - \log y ) \\
={}& \frac{1}{4} \left[ x ^s y^{1-s} \right]_0^{\frac{1}{2}}  ( \log x - \log y ) \\
= {}& \frac{1}{4} \sqrt{y}(\sqrt{x}-\sqrt{y})  ( \log x - \log y )  \,.
\end{split}\label{estbelow}
\end{align}
The inequality in the above equality chain follows from the fact that the term under the integral is positive, therefore, half of the integral can be dropped to estimate the term from below. On the lower half, $(1-s)$ can be estimated from below by $1/2$. Afterwards, the integral can be calculated explicitly.

Let now $g\in \C^1 ( \R _+ , \R)$ be  either the exponential or the square root function. 
We observe again with the fundamental theorem of calculus that
\begin{align*}
|f ( x) -f (y) |={}& |f ( g^{-1}(g(x))) - f ( g ^{-1} (g(y))) |\\={}& |\int_0^1 ( f \circ g^{-1})' ( g ( y)+ s ( g (x) - g(y))) \de s ( g (x) - g(y) )| \\
\leq {}&\int_0^1\left | (f'\circ g^{-1} )(g^{-1} )'  ( g ( y)+ s ( g (x) - g(y))) \right |\de s | g (x) - g(y) | \\
\leq {}& c | g (x) - g(y) |\,.
\end{align*}
Note that $ |(g^{-1})'(y) |\leq c (g^{-1}(y))^3  $ for $g(x)= \sqrt{x}$ and $ |(g^{-1})'(y) |\leq c (g^{-1}(y))  $ for $g(x)=\log x $.

The last inequality follows from the properties of $f$ and $g$. Note that especially the properties of $f$ enter such that the constant does not depend on $x$ and $y$. 
Choosing $g$ as the square root and the logarithm, we observe
\begin{align*}
|f (x) - f (y) |\leq c | \sqrt x - \sqrt y | \quad \text{ and} \quad| f ( x)- f (y)| \leq c | \log x - \log y| \,,
\end{align*}
such that we find with~\eqref{estbelow} 
\begin{align*}
|f(x ) - f( y)|^2  \leq \frac{c}{\sqrt{y}} (x-y-y(\log x -\log y))\,. 
\end{align*}
The second part of the estimate~\eqref{logest} can be observed by 
\begin{align*}
|f (x) - f (y) |\leq c | \sqrt x - \sqrt y | \quad \text{ and} \quad |f ( x)- f (y)| \leq  2 C  \,,
\end{align*}

\end{proof}
\begin{lemma}\label{lem:fenchel}
Let $ f \in \C^1(\R; \R)$ with $ 0 < c_f \leq f (x) \leq C_f$ and $ |f'(x)|\leq C$ for all $ x \in B_R(0)$ . Then there exists an $C>0$ such that
\begin{align*}
( f( z) -f (\tz) ) ( \theta - \tet) \leq C( ( z -\tz)^2 + \theta -\tet -\tet( \log \theta - \log \tet) ) 
\end{align*}
holds for all $ \theta  \in [0,\infty)$, $\tet \in (0,\infty)$, and $ z $, $\tz \in B_R(0)$ with $R>0$. 
\end{lemma}
\begin{proof}
%
Consider the convex, positive function $g : (-1,\infty) \ra (0,\infty)$; $g(x)=  x - \log(x+1) \geq 0$. Its Legendre--Fenchel transform is given by $ g^*(-\infty , 1)\ra (0,\infty)$; $ g^*( p) = - \log ( 1- p) - p \geq 0$. 
Young's inequality for the Legendre--Fenchel transform implies that $ ab\leq g(a) + g^*  (b)$ for $ a \in (-1,\infty) $ and $ b\in ( -\infty , 1) $ (see~\cite{ekeland}). 
The difference of the temperatures can be transformed by this function via
\begin{align}
 g \left ( \frac{\theta }{\tet}-1\right ) = \frac{\theta }{\tet}-1 - \log \frac{\theta }{\tet} = \frac{1}{\tet}\left ( \theta - \tet - \tet ( \log \theta -\log \tet)\right )\,.
\end{align}
Using Taylor's formula, we start to estimate
\begin{align*}
( f( z) - f( \tz) ) ( \theta - \tet) ={}& 4R \tet \int_0^1 | f'(\tz+s(z-\tz)) | \de s \left ( \frac{z-\tz}{4R }\right ) \left ( \frac{\theta}{\tet} - 1\right ) \\
\leq {}&4R \tet C \left ( g ^* \left ( \frac{z-\tz}{4R }\right )+ g \left ( \frac{\theta}{\tet} - 1\right ) \right )\,.
\end{align*}
The Taylor approximation of second order of  $g^*(p)= - \log ( 1- p) - p $ is given by
\begin{align*}
g^*(p)= g^*(0)+(g^*)'(0) p +\frac{1}{2} \int_0^1(1-s) (g^*)''(sp) \de s p^2 = \frac{1}{2} \int_0^1 \frac{1-s}{(1-sp)^2} \de s p^2 \leq  p^2 \,,
\end{align*}
where the first inequality is Taylor's formula of order two, the second equality follows from calculating the derivatives and the inequality follows from $p\leq 1/2$.

\end{proof}

\subsection{Relative energy inequality\label{sec:relin}}
In the following, we execute the proof of Proposition~\ref{prop:weakstrong} and therewith, Theorem~\ref{thm:weakstrong}. 
The energy inequality~\eqref{energyequality} (see Definition~\ref{def:weak}) grants  
\begin{align}
\left . \left (\int_\Omega \psi(\theta(t), z(t) ) - \theta(t) \psi_\theta (\theta(t), z(t) )\de \f x  + \int_{D} \frac{1}{2\mu(z)}| \curl \f A|^2\de \f x  \right ) \right |_0^t\quad\nonumber
  \\+ \int_0^t\left ( \int_\Omega  \frac{\mu'(z)\t z }{2\mu^2(z)} | \curl \f A |^2 \de \f x + \int_{D/ \Omega} \sigma_{\out} | \t \f A |^2 \de \f x \right )\de s   \leq{}&\int_0^t  \int_{\Sigma} \f J_s \cdot \t \f A \de \f x \de s \label{ws:enin}
 \intertext{and the  energy equality for the regular solution (see Definition~\ref{def:regular}) }
\left .\left (  \int_\Omega \psi(\tet(t), \tz(t) ) - \tet(t) \psi_\theta (\tet(t), \tz(t) )\de \f x  +\int_{D} \frac{1}{2\mu(\tz)}| \curl \tA|^2  \de \f x\right )\right |_0^t \quad  \nonumber\\
+ \int_0^t \left (\int_\Omega  \frac{\mu'(\tz)\t \tz }{2\mu^2(\tz)} | \curl \tA |^2 \de \f x + \int_{D/\Omega} \sigma_{\out} | \t \tA|^2 \de \f x \right ) \de s  ={}&   \int_0^t \int_{\Sigma} \f J_s \cdot \t \tA \de \f x \de s \label{ws:eneq}
\,.
\end{align}
The entropy inequality~\eqref{entropy} with $\vartheta=\tet$, implies that
\begin{align}
\left  .\int_\Omega  \psi_\theta(\theta , z )  \tet  \de x \right |_0^t +{}& \int_0^t \int_\Omega \tet \left ( k(\theta ,z) {|\nabla\log  \theta |^2} + \sigma (\theta ) \frac{|\t \f A |^2}{\theta }+ \tau(\theta ) \frac{|\t z|^2}{\theta } \right ) \de x \de s   \\
- {}& \int_0^t \int_\Omega \kappa(\theta ,z) {\nabla \log \theta }  \cdot \nabla \tet \de x \de s  \leq \int_0^t \int_\Omega \psi_\theta (\theta , z  ) \t \tet \de x \de s  \,.
\end{align} 
In the following, we use equation~\eqref{temp} for the regular solution several times.
We find similar to the entropy production rate by testing equation~\eqref{temp} with $-(\theta -\tet) / \tet$ that
\begin{multline}
 \int_0^t \int_\Omega (\theta -\tet) \left ( k(\tet ,\tz) {|\nabla \log \tet |^2} + \sigma (\tet ) \frac{|\t \tA |^2}{\tet}+ \tau(\tet ) \frac{|\t \tz|^2}{\tet } \right ) \de x \de s   \\
+  \int_0^t \int_\Omega \di( \kappa(\tet ,\tz) {\nabla \log \tet })  (\theta -\tet) \de x \de s  =  - \int_0^t \int_\Omega \t  \psi_\theta (\tet,\tz) (\theta- \tet) \de x \de s 
 \,.
\label{entropystrong}
\end{multline}
Note that the entropy production rate holds as an equality for the regular solution. 
Integrating equation~\eqref{temp} or testing the entropy production rate~\eqref{entropy} for the regular solution by $\tet$, we may infer
\begin{align*}
-\int_\Omega \tet \psi_\theta (\tet ,\tz) \de \f x \Big |_0^t -\int_0^t \int_\Omega \sigma(\tet ) | \t \tA   |^2 + \tau(\tet) |\t \tz|^2 \de \f x \de s   = - \int_0^t\int_\Omega  \t \tet\psi_\theta(\tet,\tz ) \de \f x \de s  
\,.
\end{align*}
Combining the previous three equations implies
\begin{multline}
-\left  .\int_\Omega\tet \psi_\theta(\tet,\tz)   \de x \right |_0^t + \int_0^t \int_\Omega (\theta -\tet)  k(\tet ,\tz) {|\nabla \log \tet |^2}+ \di (\kappa(\tet ,\tz) {\nabla\log \tet })  (\theta -\tet)    \de x \de s   \\
+ \int_0^t \int_\Omega (\theta -2\tet)  \left (\sigma (\tet ) \frac{|\t \tA |^2}{\tet}+ \tau(\tet ) \frac{|\t \tz|^2}{\tet }\right )  \de x \de s\\  = -  \int_0^t \int_\Omega \psi_\theta (\tet,\tz)  \t \tet+( \theta - \tet) \partial_t \psi_\theta ( \tet,\tz)   \de x \de s 
 \,.
\end{multline}
We turn to the two weak formulations for the Maxwell equation and the phase variable.
The fundamental theorem of calculus provides in the domain $\Omega$ that
\begin{align*}
-&\left .\int_\Omega\frac{1}{{\mu(\tz)}} \curl \tA  \cdot \left (  \curl \f A -  \curl \tA    \right ) - \frac{\mu'(\tz) }{2\mu(\tz)^2} | \curl \tA |^2 (z-\tz)  \de \f x \right |_0^t \\
&=- \int_0^t \int_\Omega\curl \left (\frac{1}{{\mu(\tz)}} \curl \tA \right ) \cdot \left (  \t  \f A -  \t  \tA    \right ) - \frac{\mu'(\tz) }{2\mu(\tz)^2} | \curl \tA |^2\t (z-\tz)  \de \f x \de s 
\\
&\quad{}- \int_0^t \int_\Omega  \t \left ( 
 \frac{1}{\mu(\tz)} \curl  \tA  \right )
  \cdot \left (  \curl  \f A -  \curl  \tA    \right ) -\t \left ( \frac{\mu'(\tz) }{2\mu(\tz)^2} | \curl \tA |^2\right ) (z-\tz)  \de \f x \de s \,,
\end{align*}
where the formula first only holds for more regular functions, but may be extended via standard density arguments. 
A rearrangement
  implies
\begin{align}
\begin{split}
    -&\left .\int_\Omega\frac{1}{{\mu(\tz)}} \curl \tA  \cdot \left (  \curl \f A -  \curl \tA    \right ) - \frac{\mu'(\tz) }{2\mu(\tz)^2} | \curl \tA |^2 (z-\tz)  \de \f x \right |_0^t \\
&=- \int_0^t \int_\Omega\curl \left ( \frac{1}{{\mu(\tz)}} \curl \tA \right ) \cdot \left (  \t  \f A -  \t  \tA    \right ) - \frac{\mu'(\tz) }{2\mu(\tz)^2} | \curl \tA |^2\t (z-\tz)  \de \f x \de s 
\\
&\quad{}+ \int_0^t \int_\Omega   \frac{\mu'(\tz)\t \tz }{{\mu^2(\tz)}} \curl \tA
  \cdot \left (  \curl  \f A -  \curl  \tA    \right ) +\t \left ( \frac{\mu'(\tz) }{2\mu(\tz)^2} \right )| \curl \tA |^2(z-\tz)  \de \f x \de s 
\\
&\quad{}- \int_0^t \int_\Omega \curl \t \tA  
  \cdot \left (  \frac{1}{\mu(z)}\curl  \f A -  \frac{1}{\mu(\tz)} \curl  \tA    \right ) - \left ( \frac{1}{\mu(z)}-  \frac{1}{\mu(\tz)}\right )  \curl \t \tA  
  \cdot \left ( \curl  \f A - \curl  \tA    \right )  \de \f x \de s 
  \\
&\quad{}+ \int_0^t \int_\Omega\frac{1}{2}\t | \curl \tA|^2 \left (\frac{1}{\mu(z)}- \frac{1}{\mu(\tz)} + \frac{\mu'(\tz)}{\mu^2(\tz)} ( z-\tz) \right )  \de \f x \de s 
\,.\end{split}\label{number}
\end{align}

In the following, we would like to test~\eqref{magnetic} for the strong solution by $\t \f A -\t \tA$. But on the domain $D/(\Omega \cap \Sigma)$ the function $\t \f A$ may not exists in $L^2$-sense, but only in the weak sense. Integrating by parts in time on the domain $D/\Omega$, we find the equality
\begin{multline}
\int_0^t \int_{\Omega} \sigma(\tet) \t \tA \cdot ( \t \f A - \t \tA ) + \curl \left ( \frac{1}{\mu(\tz) } \curl \tA \right )   \cdot ( \t \f A - \t \tA ) \de \f x \de s+\int_0^t \int_{D/\Omega} \sigma_{\out} \t \tA ( \t \f A - \t \tA) \de \f x \de s  \quad \\
+ \left  . \int_{D/\Omega} \frac{1}{\mu_{\out}} \curl \tA \cdot \left ( \curl \f A - \curl \tA \right ) \de \f x \right |_0^t \quad \\- \int_0^t \int_{D/\Omega} \t \left ( \frac{1}{\mu_{\out}} \curl \tA \right ) \cdot \left ( \curl \f A- \curl \tA \right ) \de \f x \de s = \int_0^t \int_\Sigma \f J_s \cdot \left ( \t \f A - \t \tA \right ) \de \f x \de s \,
\label{A1}
\end{multline}
first for all suitable test functions, which is then chosen to be $\f A - \tA$ in the above equality. 
Similar, we subtract from equation~\eqref{weakA} with $\varphi = \t \tA  $ equation~\eqref{magnetic} for the regular solution tested with $\t \tA$
\begin{align}
\begin{split}
0={}& \int_0^t  \int_D  ( \sigma(\theta ) \t \f A -\sigma(\tet) \t \tA) \cdot \t \tA  +\left  (\frac{1}{\mu(z) } \curl \f A - \frac{1}{\mu(\tz)} \curl \f A \right ) \cdot \curl \t \tA  \de \f x \de s 
\\
={}&
 \int_0^t  \int_\Omega  ( \sigma(\theta ) \t \f A -\sigma(\tet) \t \tA) \cdot \t \tA  +\left  (\frac{1}{\mu(z) } \curl \f A - \frac{1}{\mu(\tz)} \curl \f A \right ) \cdot \curl \t \tA  \de \f x \de s 
 \\
&+ 
 \int_0^t  \int_{D/\Omega}   \sigma_{\out}( \t \f A - \t \tA) \cdot \t \tA  +\frac{1}{\mu_{\out} }\left (  \curl \f A - \curl \f A \right ) \cdot \curl \t \tA  \de \f x \de s 
\,.
\end{split} \label{A2}
\end{align}
Notice that there are some terms occuring twice with different signs such that they will vanish, when the results are combined afterwards. The term in the second line of~\eqref{A1} corresponds to the term in the last line of~\eqref{relativeen} on $D/\Omega$. The term on the left-hand side of the last line of~\eqref{A1} corresponds to the second term in the last line of~\eqref{A2}, and the right-hand side of~\eqref{A1} corresponds to the right-hand sides of the energy inequality~\eqref{ws:enin} and the energy equality~\eqref{ws:eneq}. The sum of the second term on the left-hand side of~\eqref{A1} and the first term on the right-hand side of~\eqref{number} vanishes. The second term on the right-hand side of~\eqref{A2} together with the first term in the fourth line of~\eqref{number} vanish. 

With respect to the phase equation, we find with the fundamental theorem of calculus that
\begin{align*}
- \left . \int_\Omega \psi_z (\tu) ( z -\tz) \de \f x \right  |_0^t = {}& - \int_0^t \int_\Omega \partial_t \psi_{z} ( \tu) ( z-\tz ) + \psi_{z} ( \tu) ( \t z-\t  \tz ) \de \f x \de s  \,.
\end{align*}
Equation~\eqref{inner} for the regular functions tested with $ \t z - \t \tz$ and adding equation~\eqref{weakz} tested with $\t \tz$, and subtract equation~\eqref{inner} tested with $\t \tz$,  yields
\begin{align*}
\int_0^T\int_\Omega  \tau( \tet ) \t \tz  ( \t z - \t \tz) + \psi_{ z}(\tu) ( \t z - \t \tz)+ ( \tau(\theta) \t z - \tau (\tet) \t \tz) \t \tz +( \psi_z(\f u )- \psi_z(\tu)) \t \tz  \de \f x \de s  = 
0
\,.
\end{align*}


Inserting everything into the relative energy gives the relative energy inequality
\begin{align}
\left. \mathcal{E}(\f u(t) | \tu(t))\right |_0^t {}&+ \int_0^t\int_\Omega \tet \kappa(\theta ,z) { | \nabla\log  \theta|^2} - \kappa ( \theta , z) {\nabla \log\theta } \cdot \nabla \tet  \de \f x \de s  \notag
\\
&+ \int_0^t\int_\Omega  ( \theta - \tet) \left (\kappa ( \tet , \tz) {| \nabla \log \tet|^2} +\di( \kappa(\tet,\tz) {\nabla \log \tet })\right )  \de \f x \de s  \notag
\\
{}&+\int_0^t \int_\Omega \tet\left ( \sigma(\theta) \frac{|\t \f A |^2 }{\theta } + \tau(\theta) \frac{|\t z|^2 }{\theta } \right ) + ( \theta-2\tet) \left (\sigma (\tet ) \frac{|\t \tA |^2}{\tet}+ \tau(\tet ) \frac{|\t \tz|^2}{\tet }\right ) \de \f x \de s  \notag\\
&- \int_0^t \int_\Omega  \sigma(\tet ) \t \tA  \cdot ( \t \f A - \t \tA)  +( \sigma(\theta ) \t \f A)-\sigma(\tet) \t \tA  ) \cdot \t \tA  \de \f x \de s  \notag\\
&- \int_0^t \int_\Omega   \tau( \tet ) \t \tz  ( \t z - \t \tz)+ ( \tau(\theta) \t z - \tau (\tet) \t \tz) \t \tz \de \f x \de s   \notag
+ \int_0^t \int_{D/\Omega} {\sigma_{\out}} | \t \f A- \t \tA |^2 \de \f x 
\\
\leq{}& \int_0^t \int_\Omega \t \tet \left ( \psi_\theta (\theta ,z  ) -\psi_\theta (\tet,\tz
)  \right ) - \partial _t \psi_\theta(\tet,\tz) ( \theta - \tet)  \de \f x \de s  \notag\\
& +\int_0^t \int_\Omega  \t \tz ( \psi_z(\theta,z) - \psi_z (\tet,\tz) ) - \partial_t \psi_z (\tet,\tz) ( z -\tz)\de \f x \de s  \notag
\\
&+ \int_0^t \int_\Omega     \left ( \frac{1}{\mu(z)}-  \frac{1}{\mu(\tz)}\right )  \curl \t \tA  
  \cdot \left ( \curl  \f A - \curl  \tA    \right )   \de \f x \de s \notag
  \\
&+ \int_0^t \int_\Omega\frac{1}{2}\t | \curl \tA|^2 \left (\frac{1}{\mu(z)}- \frac{1}{\mu(\tz)} + \frac{\mu'(\tz)}{\mu^2(\tz)} ( z-\tz) \right )  \de \f x \de s \notag
\\
&+\int_0^t \int_\Omega  \frac{\mu'(\tz) }{2\mu(\tz)^2} | \curl \tA |^2\t (z-\tz) +  \frac{\mu'(\tz)\t \tz }{2\mu^2(\tz)} | \curl \tA |^2- \frac{\mu'(z)\t z }{2\mu^2(z)} | \curl \f A |^2  \de \f x \de s \notag
\\
&+ \int_0^t \int_\Omega   \frac{\mu'(\tz)\t \tz }{{\mu^2(\tz)}} \curl \tA
  \cdot \left (  \curl  \f A -  \curl  \tA    \right ) +\t \left ( \frac{\mu'(\tz) }{2\mu(\tz)^2} \right )| \curl \tA |^2(z-\tz)  \de \f x \de s 
\,.
\label{relenin}
\end{align}
For the terms in the first and the second line on the right-hand side, 
the chain rule implies
\begin{align*}
 \t \tet \left ( \psi_\theta (\theta ,z  ) -\psi_\theta (\tet,\tz
)  \right ) - \partial _t \psi_\theta(\tet,\tz) ( \theta - \tet) +  \t \tz ( \psi_z(\theta,z) - \psi_z (\tet,\tz) ) - \partial_t \psi_z (\tet,\tz) ( z -\tz) \\
=\t \tet \left ( \psi_\theta (\theta ,z  ) -\psi_\theta (\tet,\tz
) - \psi_{\theta \theta}(\tet,\tz) ( \theta - \tet)- \psi_{z\theta} (\tet,\tz) ( z -\tz)    \right ) \\
+\t \tz \left ( \psi_z (\theta , z  ) -\psi_z (\tet,\tz
) - \psi_{z,\theta}(\tet,\tz) ( \theta - \tet)- \psi_{zz} (\tet,\tz) ( z -\tz)    \right ) \,.
\end{align*}
For the last two lines of~\eqref{relenin}, we find 
\begin{align*}
&\int_0^t \int_\Omega  \frac{\mu'(\tz) }{2\mu(\tz)^2} | \curl \tA |^2\t (z-\tz) +  \frac{\mu'(\tz)\t \tz }{2\mu^2(\tz)} | \curl \tA |^2- \frac{\mu'(z)\t z }{2\mu^2(z)} | \curl \f A |^2  \de \f x \de s 
\\
&+ \int_0^t \int_\Omega   \frac{\mu'(\tz)\t \tz }{{\mu^2(\tz)}} \curl \tA
  \cdot \left (  \curl  \f A -  \curl  \tA    \right ) +\t \left ( \frac{\mu'(\tz) }{2\mu(\tz)^2} \right )| \curl \tA |^2(z-\tz)  \de \f x \de s 
\\
&{}=- \int_0^t \int_\Omega \t \log \mu(z) \left (   \frac{1}{2\mu(z)} | \curl \f A|^2 - \frac{1}{2\mu(\tz)} | \curl \tA|^2   \right ) \de \f x \de s \\
&\quad + \int_0^t \int_\Omega \t \log \mu(z) \left (   \frac{1}{{\mu(\tz)}} \curl \tA  \cdot \left (  \curl \f A -  \curl \tA    \right ) - \frac{\mu'(\tz) }{2\mu(\tz)^2} | \curl \tA |^2 (z-\tz)   \right ) \de \f x \de s \\
&\quad - \int_0^t \int_\Omega \frac{1}{\mu(\tz)} \curl \tA \cdot (\curl \f A- \curl \tA) \left ( \t ( \log \mu(z)- \log \mu(\tz))\right )  \de \f x \de s
\\
&\quad - \int_0^t \int_\Omega  \frac{\mu'(\tz) }{2\mu(\tz)} | \curl \tA|^2 ( z-\tz) \left ( \t ( \log \mu( \tz) - \log \mu(z))\right ) \de \f x \de s
\\
&\quad - \int_0^t \int_\Omega \frac{1}{2\mu(\tz)}| \curl \tA|^2 \t \left ( \log \mu (z) - \log\mu(\tz) - \partial_z \log \mu(\tz) (z-\tz) \right ) \de \f x \de s \,.
\end{align*}
The remaining part of the proof consists of handling the dissipative terms on the left-hand side of~\eqref{relenin} appropriately and estimate the remaining terms appearing on the right-hand side of~\eqref{relenin} in terms of $\mathcal{E}$ and $\mathcal{W}$. 

\subsection{Dissipative terms\label{sec:diss}}
In this subsection, we focus on the different dissipative terms. Concerning the heat dissipation, we observe
with some rearrangements
\begin{align}
\begin{split}
&\tet \kappa(\theta ,z) { | \nabla \log \theta|^2} - \kappa ( \theta , z) {\nabla \log \theta } \cdot \nabla \tet + ( \theta - \tet) \kappa ( \tet , \tz) {| \nabla\log \tet|^2} + \di ( \kappa(\tet,\tz) {\nabla \log\tet })  (  \theta -\tet) \\
&\qquad = \tet \kappa( \theta ,z)  \nabla \log \theta \cdot ( \nabla \log \theta - \nabla \log \tet) + \kappa( \tet , \tz) \left ( \theta - \tet - \tet \log \frac{\theta}{\tet}\right ) | \nabla \log \tet |^2 \\
& \qquad \quad +   \di ( \kappa(\tet,\tz) {\nabla \log\tet }) \left  (  \theta -\tet- \tet \log \frac{\theta}{\tet}\right ) \\
& \qquad \quad
 + \left ( \kappa ( \tet , \tz) | \nabla \log \tet|^2+ \di ( \kappa(\tet,\tz) \nabla \log \tet) \right ) \tet  ( \log \theta - \log \tet) \\
&\qquad = \tet \kappa( \theta ,z) |\nabla \log \theta - \nabla \log \tet|^2 + \tet( \kappa (\theta , z) -\kappa (\tet, \tz))    \nabla \log \tet ( \nabla \log \theta - \nabla \log \tet) \\
&\qquad\quad + \left ( \kappa ( \tet , \tz) | \nabla \log \tet|^2 + \di ( \kappa(\tet,\tz) \nabla \log \tet) \right )  
\left (  \theta -\tet- \tet \log \frac{\theta}{\tet}\right ) \\
& \qquad \quad +  \kappa( \tet , \tz)  \tet  \nabla \log \tet \cdot  ( \nabla \log \theta - \nabla \log \tet) 
\\& \qquad \quad
+ \left ( \kappa ( \tet , \tz) | \nabla \log \tet|^2 + \di ( \kappa(\tet,\tz) \nabla \log \tet) \right ) \tet  ( \log \theta - \log \tet) \,.
\end{split}\label{dissheat}
\end{align}
An integration-by-parts for the last term on the right-hand side of~\eqref{dissheat} implies
\begin{multline*}
\int_\Omega  \di ( \kappa(\tet,\tz) \nabla \log \tet)   \tet  ( \log \theta - \log \tet) \de \f x \\= - \int_\Omega  \kappa ( \tet , \tz) | \nabla \log \tet|^2  \tet  ( \log \theta - \log \tet) +  \kappa( \tet , \tz)  \tet  \nabla \log \tet \cdot  ( \nabla \log \theta - \nabla \log \tet)  \de \f x \,,
\end{multline*}
such that the integral over the last line of~\eqref{dissheat} vanishes. 

Estimating the second term on the right-hand side of the second equality sign in~\eqref{dissheat} by Young's inequality yields
\begin{align*}
&\int_\Omega \tet \kappa(\theta ,z) { | \nabla \log \theta|^2} - \kappa ( \theta , z) {\nabla \log \theta } \cdot \nabla \tet + ( \theta - \tet) \kappa ( \tet , \tz) {| \nabla\log \tet|^2} + \di ( \kappa(\tet,\tz) {\nabla \log\tet })  (  \theta -\tet) \de \f x  \\
&\qquad \geq  \int_\Omega  \frac{\tet}2 \kappa( \theta ,z) |\nabla \log \theta - \nabla \log \tet|^2 - c \tet | \kappa (\theta , z) -\kappa (\tet, \tz))|^2   | \nabla \log \tet|^2 \de \f x  \\
& \qquad \quad +   \int_\Omega   ( \theta - \tet - \tet(\log \theta - \log \tet) )  ( \kappa( \tet , \tz)  | \nabla \log \tet |^2 + \di ( \kappa( \tet , \tz) \nabla \log \tet) )\de \f x  \,.
\end{align*}

From now on, we discuss the dissipative terms due to Joule heating and phase transition. 
The calculations are the same for both contributions, hence, we only present them for the phase transition case
\begin{align*}
&\tau(\theta)\frac{\tet}{\theta}  \t z^2 + \tau(\tet ) \frac{\theta}{\tet} \t \tz^2 - 2 \tau(\tet) \t \tz^2 - \tau( \tet ) \t \tz  ( \t z - \t \tz)- ( \tau(\theta) \t z - \tau (\tet) \t \tz) \t \tz\\
&\qquad = \tau(\theta)\frac{\tet}{\theta}  \t z^2 + \tau(\tet ) \frac{\theta}{\tet} \t \tz^2  - (\tau( \tet ) +  \tau(\theta)) \t z  \t \tz
\\
&\qquad = \tau(\theta)\sqrt{\frac{\tet}{\theta} }  \t z\left (\sqrt{\frac{\tet}{\theta} }  \t z -\sqrt{\frac{\theta}{\tet} }  \t \tz\right ) - \tau(\tet ) \sqrt{\frac{\theta}{\tet} }  \t \tz\left (\sqrt{\frac{\tet}{\theta} }  \t z -\sqrt{\frac{\theta}{\tet} }  \t \tz\right )
\\
&\qquad = \tau(\theta)\left (\sqrt{\frac{\tet}{\theta} }  \t z -\sqrt{\frac{\theta}{\tet} }  \t \tz\right )^2  + ( \tau(\theta)-  \tau(\tet ) ) \sqrt{\frac{\theta}{\tet} }  \t \tz\left (\sqrt{\frac{\tet}{\theta} }  \t z -\sqrt{\frac{\theta}{\tet} }  \t \tz\right )\,.
\end{align*}
Applying Young's inequality, we observe
\begin{align*}
&\tau(\theta)\frac{\tet}{\theta}  \t z^2 + \tau(\tet ) \frac{\theta}{\tet} \t \tz^2 - 2 \tau(\tet) \t \tz^2 - \tau( \tet )\t  \tz  ( \t z - \t \tz)- ( \tau(\theta) \t z - \tau (\tet) \t \tz) \t \tz\\
&\qquad \geq \frac{ \tau(\theta)}{2}\left (\sqrt{\frac{\tet}{\theta} }  \t z -\sqrt{\frac{\theta}{\tet} }  \t \tz\right )^2  -c  {( \tau(\theta)-  \tau(\tet ) )^2} \frac{\theta}{\tet} |  \t \tz|^2\,.
\end{align*}
The second term on the right-hand side has to be bounded by the relative energy. Therefore it remains to adequately estimate the variable $\theta$. 
By some rearrangement, we find
\begin{multline*}
\frac{( \tau(\theta)-  \tau(\tet ) )^2}{2} \frac{\theta}{\tet} |  \t \tz|^2 = \\ \frac{( \tau(\theta)-  \tau(\tet ) )^2}{2\tet } |  \t \tz|^2 ( \theta - \tet - \tet( \log \theta - \log \tet) ) + \frac{( \tau(\theta)-  \tau(\tet ) )^2}{2\tet } |  \t \tz|^2 (\tet + \tet( \log \theta - \log \tet) )\,.
\end{multline*}
Similar calculation for the dissipative terms due to Joule heating let us conclude from~\eqref{relenin}
\begin{align}
\mathcal{E}(\f u(t) | \tu(t)) {}&+ \int_0^t\int_\Omega 
\frac{\tet}2 \kappa( \theta ,z) |\nabla \log \theta - \nabla \log \tet|^2+ \frac{\sigma(\theta ) }{2} \left |   \sqrt{\frac{\tet}{\theta} }  \t \f A-\sqrt{\frac{\theta}{\tet} }  \t \tA    \right |^2
\de \f x \de s  \notag
\\
&+ \int_0^t\int_\Omega  \frac{ \tau(\theta)}{2}\left (\sqrt{\frac{\tet}{\theta} }  \t z -\sqrt{\frac{\theta}{\tet} }  \t \tz\right )^2 
\de \f x \de s  \notag
+ \int_0^t \int_{D/\Omega} {\sigma_{\out}} | \t \f A- \t \tA |^2 \de \f x 
\\
\leq{}&   \mathcal{E}(\f u(0), \tu(0)) + \int_0^t \int_\Omega   \t \tet \left ( \psi_\theta (\theta ,z  ) -\psi_\theta (\tet,\tz
) - \psi_{\theta,\theta}(\tet,\tz) ( \theta - \tet)- \psi_{z,\theta} (\tet,\tz) ( z -\tz)    \right )\de \f x \de s  \notag
\\
&+ \int_0^t \int_\Omega\t \tz \left ( \psi_z (\theta ,z  ) -\psi_z (\tet,\tz
) - \psi_{z,\theta}(\tet,\tz) ( \theta - \tet)- \psi_{z,z} (\tet,\tz) ( z -\tz)    \right )  \de \f x \de s   \notag
\\
&+ \int_0^t \int_\Omega     \left ( \frac{1}{\mu(z)}-  \frac{1}{\mu(\tz)}\right )  \curl \t \tA  
  \cdot \left ( \curl  \f A - \curl  \tA    \right )   \de \f x \de s \notag
  \\
&+ \int_0^t \int_\Omega\frac{1}{2}\t | \curl \tA|^2 \left (\frac{1}{\mu(z)}- \frac{1}{\mu(\tz)} + \frac{\mu'(\tz)}{\mu^2(\tz)} ( z-\tz) \right )  \de \f x \de s \notag
\\&+ \int_0^t\int_\Omega  \frac{\tet}2| \kappa (\theta , z) -\kappa (\tet, \tz))|^2   | \nabla \log \tet|^2 \de \f x  \notag
\\
& + \int_0^t \int_\Omega ( \tet +  \tet(\log \theta - \log \tet) ) \left (\frac{( \tau(\theta)-  \tau(\tet ) )^2}{2\tet } |  \t \tz|^2 + \frac{(\sigma(\theta ) - \sigma(\tet))^2 }{2 \tet} \left | \t \tA \right |^2\right ) \de \f x \de s \notag  
 \\
& - \int_0^t   \int_\Omega   ( \theta - \tet - \tet(\log \theta - \log \tet) )  ( \kappa( \tet , \tz)  | \nabla \log \tet |^2 + \di ( \kappa( \tet , \tz) \nabla \log \tet) )\de \f x \de s  \notag\\
& + \int_0^t \int_\Omega
 ( \theta - \tet - \tet(\log \theta - \log \tet) )
  \left (\frac{( \tau(\theta)-  \tau(\tet ) )^2}{2\tet } |  \t \tz|^2 + \frac{(\sigma(\theta ) - \sigma(\tet))^2 }{2 \tet} \left | \t \tA \right |^2\right ) \de \f x \de s   \notag
\\
 &-  \int_0^t \int_\Omega \t \log \mu(z) \left (   \frac{1}{2\mu(z)} | \curl \f A|^2 - \frac{1}{2\mu(\tz)} | \curl \tA|^2  \right ) \de \f x \de s\notag
  \\ 
 &+  \int_0^t \int_\Omega \t \log \mu(z) \left (    \frac{1}{{\mu(\tz)}} \curl \tA  \cdot \left (  \curl \f A -  \curl \tA    \right ) - \frac{\mu'(z) }{2\mu(z)^2} | \curl \tA |^2 (z-\tz)   \right ) \de \f x \de s \notag
 \\
& - \int_0^t \int_\Omega \frac{1}{\mu(\tz)} \curl \tA \cdot (\curl \f A- \curl \tA) \left ( \t ( \log \mu(z)- \log \mu(\tz))\right )  \de \f x \de s\notag
\\
& - \int_0^t \int_\Omega \frac{\mu'(\tz) }{2\mu(\tz)} | \curl \tA|^2 ( z-\tz) \left ( \t ( \log \mu( \tz) - \log \mu(z))\right ) \de \f x \de s\notag
\\
&- \int_0^t \int_\Omega \frac{1}{2\mu(\tz)}| \curl \tA|^2 \t \left ( \log \mu (z) - \log\mu(\tz) - \partial_z \log \mu(\tz) (z-\tz) \right ) \de \f x \de s \,.
\label{secondrelin}
\end{align}
\subsection{Remaining terms\label{sec:est}}
In the following, we argue that the right-hand side can be bounded by the relative energy line by line. 

Concerning the terms in the first two lines of the right-hand side of~\eqref{secondrelin}, we observe that
\begin{align}
\begin{split}
 \t \tet& \left ( \psi_\theta (\theta ,z  ) -\psi_\theta (\tet,\tz
) - \psi_{\theta,\theta}(\tet,\tz) ( \theta - \tet)- \psi_{z,\theta} (\tet,\tz) ( z -\tz)    \right )\\
+\t &\tz \left ( \psi_z (\theta ,z  ) -\psi_z (\tet,\tz
) - \psi_{z,\theta}(\tet,\tz) ( \theta - \tet)- \psi_{z,z} (\tet,\tz) ( z -\tz)    \right )\\
={}&\tet \left ( \psi_\theta ( \theta , z) - \psi_\theta(\tet, z) - \psi_{\theta\theta}( \tet , z) (\theta -\tet) \right ) \t \log\tet + \tet\left ( \psi_z ( \theta , z) - \psi_z ( \tet ,z) - \psi_{z\theta} ( \theta -\tet) \right ) \frac{\t \tz}{\tet} \\
&+ \left ( \psi_\theta( \tet, z) - \psi_\theta (\tet, \tz) - \psi_{z\theta} ( \tet,\tz) (z-\tz) \right ) \t \tet + \left ( \psi_z( \tet,z) - \psi_z (\tet ,\tz) - \psi_{zz} (\tet ,\tz) ( z -\tz)\right ) \t \tz \\
&+ \t \log \tet\left ( \tet\psi_{\theta\theta}(\tet,z)  - \tet\psi_{\theta\theta}(\tet,\tz)\right )( \theta - \tet)   + \frac{\t z}{\tet} \left ( \tet\psi_{z\theta}( \tet,z) -\tet \psi_{z\theta}(\tet ,\tz) \right ) (\theta - \tet)\,.
\end{split}\label{thirdderi}
\end{align}
The two terms in the first line on the right-hand side can be handled in a similar way, we exemplify the  calculations for the first term.
With the fundamental theorem of calculus, we find 
\begin{align*}
\tet& \left ( \psi_\theta ( \theta , z) - \psi_\theta(\tet, z) - \psi_{\theta\theta}( \tet , z) (\theta -\tet) \right )
\\
&= \tet \int_{\tet}^\theta \psi_{\theta\theta} ( r ,z)- \psi_{\theta\theta}( \tet , z) \de r  =   \int_{\tet}^\theta \frac{\tet}{r} r \psi_{\theta\theta} ( r ,z)- \tet \psi_{\theta\theta}( \tet , z) \de r\\
&=  \int_{\tet}^\theta \frac{\tet}{r} (r \psi_{\theta\theta} ( r ,z)- \tet \psi_{\theta\theta}( \tet , z)) \de r - \tet \psi_{\theta\theta}( \tet , z ) \int_{\tet}^\theta 1 - \frac{\tet}{r}\de r \\
&=  \int_{\tet}^\theta \frac{\tet}{r} \int_{\tet}^r \partial^2_{\theta\theta} e(s,z) \de s  \de r - \tet \psi_{\theta\theta}( \tet , z ) \int_{\tet}^\theta 1 - \frac{\tet}{r}\de r \\
&\leq c \int_{\tet}^\theta 1 - \frac{\tet}{r}\de r = c\left ( \theta - \tet -\tet(\log \theta - \log\tet) \right ) \,,
\end{align*}
where we used the boundedness of the second derivative of the energy. 
The terms in the second line of~\eqref{thirdderi} are only depending on quantities that are known to be bounded, i.e., $\theta$ is not appearing. Therefore, Taylor's formula can be applied. We exemplify the calculations again for the first term:
\begin{align*}
| \psi_\theta( \tet, z) - \psi_\theta (\tet, \tz) - \psi_{z\theta} ( \tet,\tz) (z-\tz) |\leq  c | z -\tz |^2 \,.
\end{align*}
To handle the last line of~\eqref{thirdderi}, we use Lemma~\ref{lem:thetae} and that $z\mapsto \tet\psi_{\theta\theta}( \tet ,z) $ as well as  $z\mapsto \tet\psi_{\theta z}( \tet ,z) $ are bounded regular functions, i.e., $\psi $ is locally trice differentiable to find
\begin{align*}
\left ( \tet\psi_{\theta\theta}(\tet,z)  - \tet\psi_{\theta\theta}(\tet,\tz)\right )( \theta - \tet ) \leq C ( ( z-\tz)^2 + \theta - \tet -\tet(\log \theta - \log \tet) ) \,.
\end{align*}

Due to the smoothness assumptions on $\mu$ (see~\eqref{condmu}), we find with Taylor's formula applied to 
$1/\mu$ that
\begin{align*}
\left |\frac{1}{\mu(z)}- \frac{1}{\mu(\tz)}\right |&\leq{}
\left | \int_0^1 \frac{\mu'(\tz+s(z-\tz))}{\mu^2(\tz+s(z-\tz))} \de s \right  | |z-\tz| \leq 
c | z - \tz| \,,\\
\left | \frac{1}{\mu(z)}-\frac{1}{\mu(\tz)}- \partial_{z} \left (\frac{1}{\mu(\tz)}\right ) (z-\tz) \right | &\leq{}
\left | \int_0^1 \partial_{zz}^2  \left (\frac{1}{\mu(\tz+ s(z-\tz))}\right )  \de s\right | ( z -\tz)^2 \leq 
c|z-\tz|^2\,.
\end{align*}
With Young's inequality, we find the boundedness of the third and forth line on the right-hand side of~\eqref{secondrelin} in terms of the relative energy. 
For the heat conduction, we observe that 
\begin{align*}
| \kappa(\theta ,z) -\kappa (\tet,\tz) |^2 \leq 2 | \kappa (\theta ,z) - \kappa (\tet,z) |^2 + 2 |\kappa(\tet, z) - \kappa (\tet,\tz) |^2\,.
\end{align*} 
The second term on the right-hand side can be estimated by Taylor's formula, i.e., $| \kappa(\tet, z) - \kappa (\tet,\tz) |^2 \leq  \| \partial_z \kappa\|_{\C^1}^2 | z - \tz|^2$,  and the first one by Lemma~\ref{lem:est} such that terms in line five on the right-hand side of~\eqref{secondrelin} can be estimated by $\mathcal{E}$. 
Lemma~\ref{lem:est} also helps to estimate line six on the right-hand side of~\eqref{secondrelin}.
Due to Proposition~\ref{prop:pos} as well as the definition of the relative energy~\eqref{relativeen}, line seven to ten are  bounded by the relative energy. 
For the last three lines, the regularity of $\mu$ implies that Taylor's formula can be applied to $\log \mu $, which gives
 \begin{align*}
\log\mu(z)-\log\mu(\tz) &= \int_0^1 \frac{\mu'(\tz+s(z-\tz))}{\mu(\tz+s(z-\tz))} \de s ( z-\tz)\,,
\\
\log \mu (z) - \log\mu(\tz) - \partial_z \log \mu(\tz) (z-\tz)  &= \int_0^1 \partial^2_{zz} \log(\mu( \tz + s(z-\tz)) \de s (z-\tz)^2\,.
\end{align*}
We observe that all terms  in the last three lines of~\eqref{secondrelin} can be estimated using young's inequality. It remains to estimate the term $ | \t z - \t \tz|^2 $. Therefore, the phase equation is used~\eqref{weakz}.
\begin{align*}
| \t z - \t \tz|^2 = | \psi_z( \theta , z) - \psi_z ( \tet, \tz) |^2 \leq 2 | \psi _ z( \theta ,z) - \psi _ z( \tet ,z) |^2 + 2 | \psi_z ( \tet, z) - \psi_z( \tet ,\tz) |^2 \,.
\end{align*}
For the second term on the right-hand side, we immediately find
\begin{align*}
| \psi_z ( \tet, z) - \psi_z( \tet ,\tz) |^2 \leq  \left  ( \int_0^1 \psi_{zz} ( \tet , \tz +s( z-\tz) ) \de s \right )^2 (z-\tz)^2 \leq c  (z-\tz)^2\,.
\end{align*}
For the first one, we observe that
\begin{multline}
\int_{\sqrt{\tet} }^{\sqrt{\theta}} 2s\psi_{z\theta}( s^2 , z) \de s = \psi_z \left ( \left(\sqrt{\theta}\right)^2 , z\right ) - \psi_z\left ( \left(\sqrt{\tet}\right)^2 ,z\right ) \\
= \psi_z(\theta ,z) - \psi_z(\tet,z) 
 = \psi_z\left (\exp({\log\theta}), z\right ) - \psi_z \left ( \exp({\log \tet}), z\right ) = \int_{\log\tet} ^ {\log \theta} \exp(s )\psi_{\theta z}( \exp(s), z) \de s \,.
\end{multline}
The bounds on $\psi_{z\theta} $ guarantee 
\begin{multline*}
  | \psi _ z( \theta ,z) - \psi _ z( \tet ,z) |^2 = \left ( \int_{\sqrt{\tet} }^{\sqrt{\theta}} 2s\psi_{z\theta}( s^2 , z) \de s\right ) \left ( \int_{\log \tet} ^ {\log \theta} \exp(s) \psi_{\theta z}(\exp( s), z) \de s\right ) \\ \leq c (\sqrt{\theta } - \sqrt{\tet} ) ( \log \theta -\log \tet) 
\end{multline*}
such that we conclude as in the proof of Lemma~\ref{lem:est} 
\begin{align*}
| \t z - \t \tz |^2 \leq c ( | z - \tz|^2 + \theta - \tet -\tet(\log \theta - \log \tet) ) \,.
\end{align*}

Putting everything together, we observe
\begin{align*}
\mathcal{E}(\f u (t)| \tu(t)) + \int_0^t \mathcal{W}(\f u | \tu) \de s  \leq \mathcal{E}(\f u(0)| \tu(0)) +c \int_0^t( \|\t \log \tet\|_{L^\infty}+ \| \t \tz\|_{L^\infty}+\| \curl\t \f\tA \|_{L^\infty} ) \mathcal{E}(\f u| \tu) \de s   
\\
+ c\int_0^t  \left (\| \nabla \log \tet\|_{L^\infty}^2 + \| \di ( \kappa(\tet,\tz) \nabla \log \tet) \|_{L^\infty} + \| \t \log \mu(z) \|_{L^\infty} +\| \curl \tA  \|_{L^\infty} ^2 \right ) \mathcal{E}(\f u | \tu) \de s  \,.
\end{align*}
Gronwall's estimate provides the assertion. 
\small
\bibliographystyle{abbrv}

\end{document}